\documentclass{amsart}
\usepackage[pdftex]{graphicx}
\usepackage{amsmath}
\usepackage{amssymb}
\usepackage{amsthm}
\usepackage{amsfonts}
\usepackage{wrapfig}
\usepackage{relsize}
\usepackage{setspace}
\usepackage{fancyhdr}
\usepackage[utf8]{inputenc}
\usepackage{mathtools}
\usepackage{esint}
\usepackage{appendix}
\usepackage{mathrsfs}
\usepackage{comment}
\usepackage{epigraph}
\usepackage{tikz}
\usetikzlibrary{arrows}
\usepackage{cite}
\usepackage{listings}
\usepackage{soul}
\usepackage[margin=0.85in]{geometry}
\usepackage{xcolor}

\newcommand{\R}[1]{\ensuremath{\mathbb{R}^{#1}}}
\newcommand{\dd}{\,{\rm d}}
\numberwithin{equation}{section}

\theoremstyle{plain}
\newtheorem{theorem}{Theorem}[section]
\newtheorem{proposition}[theorem]{Proposition}
\newtheorem{lemma}[theorem]{Lemma}

\newtheorem{rem}[theorem]{Remark}

\theoremstyle{definition}
\newtheorem{definition}[theorem]{Definition}
\newtheorem{example}[theorem]{Example}

\begin{document}

\title[Computation of Young-Measure Solutions]{Numerical Approximation of Young Measure Solutions to\\ Parabolic Systems of Forward-Backward Type}

\author{Miles Caddick}
\author{Endre S\"uli}

\address{Mathematical Institute, University of Oxford, Andrew Wiles Bldg., Woodstock Road, Oxford OX2 6GG, UK}
\email{miles.caddick@maths.ox.ac.uk}
\address{Mathematical Institute, University of Oxford, Andrew Wiles Bldg., Woodstock Road, Oxford OX2 6GG, UK}
\email{endre.suli@maths.ox.ac.uk}

\begin{abstract}
This paper is concerned with the proof of existence and numerical approximation of large-data global-in-time Young measure solutions to initial-boundary-value problems for multidimensional nonlinear pa\-rabolic systems of forward-backward type of the form $\partial_t u - \mbox{div}(a(Du)) + Bu = F$, where $B \in \mathbb{R}^{m \times m}$, $Bv \cdot v \geq 0$ for all $v \in \mathbb{R}^m$, $F$ is an $m$-component vector-function defined on a bounded open Lipschitz domain $\Omega \subset \mathbb{R}^n$, and $a$ is a locally Lipschitz mapping of the form $a(A)=K(A)A$, where $K\,:\, \mathbb{R}^{m \times n} \rightarrow \mathbb{R}$. The function $a$ may have a nonstandard growth rate, in the sense that it is permitted to have unequal lower and upper growth rates. Furthermore, $a$ is not assumed to be monotone, nor is it assumed to be the gradient of a potential. Problems of this type arise in mathematical models of the atmospheric boundary layer and fall beyond the scope of monotone operator theory. We develop a numerical algorithm for the approximate solution of problems in this class, and we prove the convergence of the algorithm to a Young measure solution of the system under consideration.
~\\
\begin{center}
\textit{\normalsize{Dedicated to Academician Professor Gradimir Milovanovi\'{c} on the occasion of his 70th birthday.}}
\end{center}
\end{abstract}

\maketitle

\section{Introduction}\label{sec:1}
The paper is concerned with the numerical approximation of Young measure solutions to initial-boundary-value problems for
nonlinear multidimensional parabolic systems of forward-backward type, where the existence of a weak solution cannot in
general be guaranteed, because the nonlinearity in the equation is neither monotone, nor globally Lipschitz, nor indeed is it the gradient of a potential. Nonlinear parabolic systems of this type arise in certain mathematical models of the atmospheric boundary layer and, to date, there have been no attempts at the rigorous mathematical analysis of numerical methods for their approximate solution.

The systems of nonlinear parabolic partial differential equations considered here are of the following form:
\begin{equation}\label{modelprob1}
\partial_tu-\mathrm{div}\left(a(Du)\right)+Bu=F,
\end{equation}
with $u:[0,T]\times \overline{\Omega}\rightarrow\R{m}$, where $\Omega\subset\R{n}$ is a bounded open Lipschitz domain. Here $Du$ denotes the spatial gradient of $u$, with $(Du)_{ik}= \partial u_i/\partial x_k$, $i=1,\dots,m$, $k=1,\dots,n$. Throughout we will take real numbers $p_i$, for $i=1,\ldots,m$, that satisfy
\[p_i>\max\bigg\{1, \frac{2n}{n+2}\bigg\}.\]
We define
\[p:=\min_{i=1,\ldots,m} p_i \qquad \mbox{and}\qquad q:=\max_{i=1,\ldots,m} p_i,\]
and require that
\begin{align}\label{e:qp1}
q-p<1.
\end{align}
We shall also assume that
\begin{equation}\label{modelrhs}
F\in L^{p'}\!(Q_T;\R{m})\cap L^2(Q_T;\R{m}),
\end{equation}
with $\frac{1}{p} + \frac{1}{p'}=1$, where $Q_T:=(0,T) \times \Omega$, and that
\begin{equation}\label{b-positive}
Bv\cdot v\geq 0\qquad \forall\, v \in \mathbb{R}^m,
\end{equation}
where $B\in \R{m\times m}$ is a constant matrix. All of our results extend directly to the case where $B \in L^\infty(Q_T;\R{m\times m})$, without assuming \eqref{b-positive}, however a more general supposition of this kind
would contribute no new insight, so for the sake of clarity of the exposition we shall continue to assume that $B \in \R{m\times m}$ is a constant matrix satisfying \eqref{b-positive}. The system \eqref{modelprob1} is supplemented with the initial condition
\begin{equation}\label{modelinitdata}
u(0,\cdot)=u_0(\cdot)\in L^2(\Omega;\R{m}),
\end{equation}
and the homogeneous Dirichlet boundary condition
\begin{equation}\label{modelboundarydata}
u|_{(0,T]\times\partial\Omega}=0.
\end{equation}
The function $a:\mathbb{R}^{m \times n} \rightarrow \mathbb{R}^{m \times n}$ is assumed to be a locally Lipschitz mapping of the form
\begin{equation}\label{modelstructure}
a(A)=K(A)A,
\end{equation}
where $K:\R{m \times n}\rightarrow\R{}$ is continuous, and there exist positive constants $c_0$ and $c_1$, with $c_0 \leq  c_1$, such that
\begin{equation}\label{modelgrowth}c_0\sum_{i=1}^m(\mu_i^2+|A_i|^2)^\frac{p_i-2}{2}\leq K(A)
\leq c_1\sum_{i=1}^m(\mu_i^2+|A_i|^2)^\frac{p_i-2}{2} \qquad \forall\, A \in \mathbb{R}^{m \times n},
\end{equation}
with $p_i>\max\{1,\frac{2n}{n+2}\}$ for all $i=1,\ldots,m$, and $\mu_i$ is a constant defined by
\begin{equation}\label{modelconstants}
\mu_i
\begin{cases}
\in\R{} \quad \text{if}\ p_i\geq2,\\ \neq0 \quad \,\text{if}\ 1<p_i<2,
\end{cases}
\end{equation}
and where $|A_i|$ denotes the Euclidean norm of the $i$-th row of the matrix $A$. More generally, $|\cdot|$ will signify the absolute value of a real number, the Euclidean norm of a vector, or the Frobenius norm of a matrix: its precise meaning will be clear from the context. We emphasize that we are making no assumptions here about the monotonicity of the mapping $a$, nor will we assume that $a$ is the gradient of a certain $C^1$-function. The results of the paper can be easily adapted to systems in which the form of $a$ is slightly more general: for example, instead of \eqref{modelstructure}--\eqref{modelconstants} one can assume that $|a(x)| \leq c_0(1+|x|^{q-1})$ and $a(x)\cdot x \geq c_1|x|^p$, with $p>\max\{1,\frac{2n}{n+2}\}$ and $p\leq q < p+1$.

In order to motivate the weak structural assumptions on the function $a$ stated above we shall now present some examples.
\begin{example}
In \cite{BeCu} the authors consider a system of equations modelling the behaviour of the atmospheric boundary layer, which served as our original motivation for the present study. In that model, $n=1$, $m=3$, $\Omega = (0,1)$, and the function $K$, referred to as a \textit{stability function},  is defined as follows:
\begin{align}\label{stabfunc}
K(\partial_xu)=\left\{
\begin{array}{cc}
\frac{(|\partial_xu_1|^2+|\partial_xu_2|^2)^\frac{3}{2}}{|\partial_xu_1|^2+|\partial_xu_2|^2+\partial_xu_3}& \text{if}\ \partial_xu_3>0,\\
~&~\\
\sqrt{|\partial_xu_1|^2+|\partial_xu_2|^2-\partial_xu_3}& \text{if}\ \partial_xu_3\leq0.
\end{array}
\right.
\end{align}
The resulting system of PDEs for the vector of unknowns $u(t,x) = (u_1(t,x),u_2(t,x),u_3(t,x))^{\rm T}$ is
\begin{align*}
\partial_tu-\partial_x(K(\partial_xu)\partial_xu)+ Bu&=F\ \ \ \ \ \ \text{for}\ (t,x)\in(0,T]\times(0,1),\\
u(t,1)&=f(t)\ \ \ \text{for}\ t\in(0,T],\\
u_1(t,0)&=g_1(t)\ \ \text{for}\ t\in(0,T],\\
u_2(t,0)&=g_2(t)\ \ \text{for}\ t\in(0,T],\\
\partial_xu_3(t,0)&=h(t)\ \ \ \text{for}\ t\in(0,T],
\end{align*}
subject to the initial condition
\begin{equation*}%\label{initialcond}
u(0,x)=u_0(x)\ \ \ \ \text{for}\ x\in(0,1).
\end{equation*}
Here, $u:[0,T]\times [0,1]\rightarrow\R{3}$ is a vector with components $u=(u_1,u_2,u_3)^{\rm T}$, $B$ is the skew-symmetric matrix
\[B=\begin{pmatrix}0& 1&0\\ -1&0&0 \\ 0&0&0 \end{pmatrix}\]
and $F$ is a forcing term given by the constant vector
\begin{equation*}
F(t,x)=\begin{pmatrix}-V\\U\\0\end{pmatrix}.
\end{equation*}
What is immediately seen is that, given the above definition of the function $K$, the product $K(\partial_xu)\partial_xu$ lacks a coercivity estimate in a Sobolev space on the range $\partial_xu_3>0$, because the mapping $\xi \in \R{3} \mapsto a(\xi) := K(\xi)\xi \in \R{3}$ is not coercive in the range $\xi_3>0$; indeed, for $\xi=(1,0,\xi_3)^{\rm T}$ with $\xi_3>0$, $a(\xi)\cdot \xi/|\xi| \rightarrow 1$ as $\xi_3 \rightarrow +\infty$. Another difficulty arises from the fact that the function $\xi \mapsto a(\xi):=K(\xi)\xi$ is not monotone on $\R{3}$, i.e., it is not true that $(a(\xi)-a(\eta)):(\xi - \eta) \geq 0$ for all $\xi, \eta \in \R{3}$, and therefore monotone operator theory cannot be applied in this setting.
To see that the mapping $\xi \mapsto a(\xi):=K(\xi)\xi$ is not monotone on $\R{3}$, take $\xi=(0.035,0,-0.01)^{\rm T}$ and $\eta=(0.05,0,0)^{\rm T}$; then, computing $(a(\xi)-a(\eta))\cdot(\xi-\eta)$ will a give a negative number. Similarly taking $\xi=(-0.2, -0.1, 0.2)^{\rm T}$ and $\eta=(-0.1,0,0.5)^{\rm T}$ will result in a negative number. Finally to show that the vector field $a$ is not a potential, consider the contour integral of $a$ around the circle $C$ parametrized by $r(\theta)=(1+\cos(\theta),0,\sin(\theta))$, $\theta \in [0,2\pi)$. One can show that this is nonzero, and so $a$ is not a conservative vector-field, and thus it is not a potential.
\end{example}
\begin{example}
For $n=1$ and $m=3$ consider the function $K$ given by $\xi \in \R{3}\mapsto K(\xi)=\sqrt{\xi_1^2+\xi_2^2+|\xi_3|}$. Then we are in the regime $\partial_x u_3 \leq 0$ from \eqref{stabfunc}. One has the existence of constants $c_0$ and $c_1$ such that
\[
c_0(|\xi_1|+|\xi_2|+\sqrt{|\xi_3|})\leq K(\xi)\leq c_1(|\xi_1|+|\xi_2|+\sqrt{|\xi_3|})
\]
through bounding and using the equivalence of norms on \R{3}. This function $K$ is therefore covered by our assumptions.
We are unfortunately unable to say anything about the region from \eqref{stabfunc} where $\partial_xu_3>0$, as our analysis does not cover the case $p=1$.
\end{example}

As will transpire from the discussion that follows, under the stated hypotheses on the function $a$ we are unable to prove
the existence of a weak solution to the problem \eqref{modelprob1}--\eqref{modelconstants} under consideration and have to weaken the notion of solution to be able to show its existence. We shall therefore, instead, show the existence of a Young measure solution in the sense of the next definition, and will then consider the numerical approximation of such Young measure solutions.

\begin{definition}\label{measuresol}
We say that a pair $(u,\nu)$, where $u$ is a function with
\[ u \in L^\infty(0,T;L^2(\Omega;\R{m}))\cap L^p(0,T;W^{1,p}_0(\Omega;\R{m}))\quad\mbox{with}\quad Du_i \in L^{p_i}(0,T;L^{p_i}(\Omega;\R{m\times n})),\quad i=1,\dots,m,\]
and
\[\partial_tu\in L^{\hat{q}'}\!(0,T;W^{-1,\hat{q}'}\!(\Omega;\R{m})),\]
and $(t,x) \in Q_T \mapsto \nu_{t,x}$ is an $\R{m \times n}$-valued Young measure, such that
\[ \langle\nu,a\rangle \in L^{\hat{q}'}\!(Q_T;\R{m\times n}),\]
is called a \textit{Young measure solution} to the problem described by \eqref{modelprob1}, \eqref{modelinitdata} and \eqref{modelboundarydata} if
\[\int_0^T\langle\partial_tu,\phi\rangle_{(W_0^{1,\hat{q}}(\Omega;\R{m}),W^{-1,\hat{q}'}\!(\Omega;\R{m}))}+\int_\Omega\langle \nu_{t,x},a\rangle : D\phi\ + Bu\cdot\phi{\dd}x{\dd}t=\int_0^T\int_\Omega F\cdot \phi{\dd}x{\dd}t\]
for all $\phi\in L^{\hat{q}}(0,T;W_0^{1,\hat{q}}(\Omega;\R{m}))$, where $\hat{q}:=\max\{q,2\}$, $\hat{q}'=\frac{\hat{q}}{\hat{q}-1}$, and the initial condition is satisfied in the sense that $u(0,x)=u_0(x)$ for almost every $x\in\Omega$.
\end{definition}

\begin{comment}
\begin{rem} The proofs and the main results of this paper can be easily adapted to the case when, instead of assumptions
\eqref{modelstructure}--\eqref{modelconstants}, the nonlinearity
$a$ satisfies $a(A)\cdot A\geq c_1|A|^p$,  $|a(A)|\leq c_2|A|^{q-1}+c_3$, the function $a$ is locally Lipschitz continuous, \eqref{A-limit} holds, and $q-p<1$, $\max\left\{1, \frac{2n}{n+2}\right\} < p \leq  q < \infty$.
\end{rem}
\end{comment}

Measure-valued solutions and Young measure solutions to parabolic partial differential equations have been studied by a number of authors. Frehse \& Specovius-Neugebauer, in \cite{FreSpec}, showed the existence of H\"older continuous Young measure solutions to two-dimensional nonmonotone quasilinear problems exhibiting quadratic growth (that is, $p=q=n=2$), under the assumption that their nonlinearity $a$ was expressible as
\begin{equation}\label{potential}
a=\nabla\Phi
\end{equation}
for some $C^1$ function $\Phi$. Our terminology \textit{Young measure solution} follows that of Frehse \& Specovius-Neugebauer in \cite{FreSpec}; in particular, we have consciously avoided referring to the solutions considered here as \textit{measure-valued} solutions, as the function $u$, whose existence we prove, is still a real-valued function with spatial Sobolev regularity; however, instead of being a standard weak solution, the function $u$ satisfies the PDE in the sense of gradient Young-measures. Young measure solutions to forward-backward parabolic problems were studied by Demoulini in \cite{Demo}: under the assumptions that $m=1$, $p=q=2$ and that \eqref{potential} is satisfied,  a sequence of approximating solutions was constructed by means of minimizing an integral, which involves the convexification of $\Phi$. Using this construction, a family of Young measures was generated displaying an independence property, from which it was possible to deduce a uniqueness result regarding the function $u$ (not regarding the family of Young measures). From this, a weak-strong uniqueness result follows if the classical solution, which is assumed to exist, satisfies additional constraints on its gradient. More recently, in Thanh \cite{Thanh} and some of the references contained therein, in particular \cite{ThanhSmaTes}, possible regularizing terms are discussed for forward-backward parabolic equations, the inclusion of which allows for solutions with higher regularity to be obtained than those considered here; Thanh also obtained results concerning the long-time behaviour of certain Young measure solutions, and (in the case of an equation in one space dimension) the support of the Young measure appearing in the Young measure solutions. However, once again, that work makes the assumption that \eqref{potential} holds, and assumes more restrictive growth conditions on $a$ than those considered here. Even more recently, Kim \& Yan in \cite{KimYan} have expanded their earlier work in \cite{KimYan2015} to show the existence of infinitely many Lipschitz solutions (and under further assumptions, the existence of a unique classical solution) for particular classes of forward-backward parabolic equations. Although their work allows one to consider stronger notions of solution than the Young measure solutions that we shall discuss here, many of their structural assumptions regarding the nonlinearities are incompatible with our setting. Furthermore, their focus is solely on scalar equations, whereas we are concerned with systems here.

From the computational point of view, Carstensen \& Roub\'i\v{c}ek \cite{CarsRou} considered the numerical approximation of Young measures arising in certain minimization problems in the calculus of variations. As we lack this variational structure ourselves (in this work we are not assuming \eqref{potential}), neither of the methods in \cite{Demo} or \cite{CarsRou} will be suitable in our setting. In \cite{BFG} the authors consider numerical experiments for particular forward-backward equations that have a regularizing term added, but only for small values of the regularizing parameter and do not consider numerical schemes for the limiting case. We mention in passing that for the Keller--Segel system, which is a coupled parabolic-elliptic system of partial differential equations, a stochastic interacting particle approximation of global-in-time measure-valued solutions in two space dimensions was considered by Ha\v{s}kovec \& Schmeiser in \cite{HS}.

The aim of this work is to develop and analyze a numerical scheme for the approximation of Young measure solutions of parabolic systems of equations under minimal assumptions on the form taken by the nonlinearity $a$, and under no restriction on the dimension of the domain $\Omega$. We make the point here that the scheme discussed in this article will not allow us to compute the Young measures appearing in the references \cite{Demo,FreSpec,CarsRou,Thanh,BFG} cited above as taking different approximating sequences to the same function will potentially result in different families of Young measures being generated. Furthermore, without a concept of uniqueness, there is no guarantee that the functions appearing in the Young measure solutions will even be the same.

The paper is structured as follows. In the next section we shall collect miscellaneous results that are used in the subsequent analysis. In Section \ref{sec:3} we will show the existence of a large-data global-in-time Young measure solution to the system \eqref{modelprob1}--\eqref{modelconstants} under consideration by performing a spatial finite element approximation and passing to the limit in the spatial discretization parameter. In Section \ref{sec:4}, we shall also discretize with respect to the temporal variable and show the convergence of the fully discrete scheme to a Young measure solution of the problem. We note here, however, that the results from Section \ref{sec:3} and Section \ref{sec:4} only yield to us indirect information about the measure $\nu$ through its action on functions belonging to a particular function space. This is the motivation behind Section \ref{sec:5}, in which, stimulated by recent contributions by Fjordholm et al. \cite{FMT} and \cite{FKMT}, in particular, we shall develop a numerical algorithm for the computation of Young measure solutions which allows for a more direct approximation of the Young measure $\nu$, and will prove that the algorithm converges to a Young measure solution of the problem under consideration.

\section{Preliminary Definitions and Results}\label{sec:2}

\subsection{Function Spaces}

We begin by defining, for $r > 0$, the function space
\[E_{r}:=\left\{g\in C(\R{m \times n};\R{m}): \lim_{|A|\rightarrow\infty}\frac{g(A)}{1+|A|^{r}}\ \text{exists}\right\},\]
which will be of use to us later. On this space we consider the norm
\[\|g\|_{E_{r}}=\sup_{A\in\R{m \times n}}\frac{|g(A)|}{1+|A|^{r}}.\]
\begin{lemma}\label{propsofF}
The following statements hold true for $r > 0$:
\begin{enumerate}
\item The space $E_{r}$ is separable;
\item The space $E_{r}$ is metrizable.
\end{enumerate}
\end{lemma}
\begin{proof}
For the first property we refer to Kinderlehrer \& Pedregal \cite{KinPed2}. The second follows by taking the metric induced by the norm $\|\cdot\|_{E_{r}}$.
\end{proof}
\begin{rem}\label{rem:sep}
We remark here that $E_{r}$ is a subspace of the larger space
\[\tilde{E}_{r}:=\left\{g\in C(\R{m \times n};\R{m}):\sup_{A\in\R{m \times n}}\frac{|g(A)|}{1+|A|^{r}}<\infty\right\},\]
which is not separable. For example, under the assumption \eqref{modelgrowth}, $a \in \tilde{E}_{q-1}$, and for any
$\alpha \in (1,\frac{p}{q-1})$ we have that $0<q-1 < \frac{p}{\alpha}<p$, and therefore $a \in E_{\frac{p}{\alpha}}$.
\end{rem}

As $C_0(\R{m \times n};\R{m})\subset E_{r}$ for all $r > 0$, we have the following nesting of dual spaces for all $r > 0$:
\[E_{r}'\subset \mathcal{M}(\R{m \times n};\R{m})= C_0(\R{m \times n};\R{m})'.\]
Since the dual of a normed space is a Banach space, we have that $E_{r}'$ is a Banach space for all $r > 0$, as we can equip $E_{r}$ with the norm induced by its metric. In particular, thanks to the Banach--Alaoglu Theorem, any bounded sequence in $E_{r}'$ has a weakly-$\ast$ convergent subsequence.

\subsection{Young Measures}

We recall here from \cite{Ball} the so-called Fundamental Theorem for Young Measures, together with some related remarks.

\begin{theorem}\label{FTYM}
Let $\Omega\subset\R{n}$ be Lebesgue measurable, let $\mathfrak{K}\subset\R{m}$ be closed and let
$z_j:\Omega\rightarrow\R{m}$, $j=1,2,\ldots,$ be a sequence of Lebesgue measurable functions satisfying
\[\lim_{j\rightarrow\infty}|\{x\in\Omega:\ z_j(x)\notin U\}|=0\]
for any open neighbourhood $U$ of $\mathfrak{K}$ in \R{m}. Then, there exists a subsequence $\{z_{j_k}\}$ of $\{z_j\}$ and a family of positive measures on \R{m}, $\{\nu_x\}_{x\in\Omega}$, depending measurably on $x\in\Omega$, such that:
\begin{enumerate}
\item $\|\nu_x\|_M:=\int_{\R{m}}{\rm d}\nu_x\leq1$ for almost every $x\in\Omega$;
\item $\mathrm{supp}\ \nu_x\subset \mathfrak{K}$ for almost every $x\in\Omega$; and
\item $f(z_{j_k})\overset{\ast}{\rightharpoonup}\langle\nu_x,f\rangle=\int_{\R{m}}f(\lambda){\dd}\nu_x(\lambda)$ in $L^\infty(\Omega)$
for each continuous function $f:\R{m}\rightarrow\R{}$ satisfying $\lim_{|\lambda|\rightarrow\infty}f(\lambda)=0$.
\end{enumerate}
Suppose further that $\{z_{j_k}\}$ satisfies the following tightness condition:
\[\lim_{s\rightarrow\infty}\sup_{k}|\{x\in\Omega\cap B_R(0):\ |z_{j_k}(x)|\geq s\}|=0,\]
for every $R>0$. Then, we have that $\|\nu_x\|_M=1$ for almost every $x\in\Omega$, and, for any measurable subset $A$ of the set $\Omega$,
\[f(z_{j_k})\rightharpoonup\langle\nu_x,f\rangle\]
in $L^1(A)$ for any continuous function $f:\R{m}\rightarrow\R{}$ such that $\{f(z_{j_k})\}$ is sequentially weakly relatively compact in $L^1(A)$.
\end{theorem}
\begin{rem}\label{rem:1}
If $\Omega$ is bounded and the sequence $\{z_j\}$ is bounded in $L^p(\Omega;\R{m})$ for some $p\in(1,\infty)$, then we obtain from the theorem the existence of a family of probability measures $\{\nu_x\}_{x \in \Omega}$ and a subsequence $z_{j_k}$ such that
\[f(z_{j_k})\rightharpoonup \langle\nu,f\rangle\]
in $L^r(\Omega)$, where $f:\R{m}\rightarrow\R{}$ is a continuous function satisfying
\[|f(\lambda)|\leq c(1+|\lambda|^{s})\]
with $1<r<\frac{p}{s}$ and $s>0$.
\end{rem}
\begin{rem}\label{rem:2}
The assumption that
\[\lim_{j\rightarrow\infty}|\{x\in\Omega:\ z_j(x)\notin U\}|=0\]
for any open neighbourhood $U$ of $\mathfrak{K}$ is only used to prove that {\rm (2)} in Theorem \ref{FTYM} holds, and is not needed to obtain any of its other conclusions. The proof of this theorem from \cite{Ball} is still valid if we take $\mathfrak{K}=\R{m}$, only it then provides no information about the support of the Young measure $\nu$. Therefore, if we are not concerned with the support of the Young measure $\nu$ we may always take $\mathfrak{K}=\R{m}$ in order to apply this theorem.
\end{rem}

\subsection{Miscellaneous Results}
We recall the following result from Strauss \cite{Strauss} (cf. also Lions \& Magenes \cite{Lions-Magenes}, Lemma 8.1, Ch. 3, Sec. 8.4).
\begin{lemma}\label{lem-Cw-Cw*}
Suppose that $X$ and $Y$ are Banach spaces. Assume that the space $X$ is reflexive and is continuously embedded in the space $Y$\!; then,
\[
L^\infty(0, T;X) \cap C_w([0, T];Y) = C_w([0, T];X),\]
where $C_w([0,T];X)$ and $C_w([0,T];Y)$ denote the spaces of weakly continuous functions from $[0,T]$ into $X$ and $Y$, respectively.
\end{lemma}

We shall also require the following consequence of the Arzel\`{a}--Ascoli theorem.

\begin{lemma}\label{ArzAscApp}
Let $\{u^j\}$ be a sequence of functions with the following properties:
\begin{enumerate}
\item $\{u^j\}$ is bounded in $L^\infty(0,T; L^2(\Omega;\R{m}))$
and
\[ u^j\overset{\ast}\rightharpoonup u\ \text{in}\ L^\infty(0,T; L^2(\Omega;\R{m}));\]
\item $\{\partial_tu^j\}$ is bounded in $L^{s'}\!(0,T;W^{-1,s'}\!(\Omega;\R{m}))$ for some $s \in (1,\infty)$,
and
\[\partial_tu^j\rightharpoonup\partial_tu\ \text{in}\ L^{s'}\!(0,T;W^{-1,s'}\!(\Omega;\R{m})).\]
\end{enumerate}
Then, $u^{j_k}, u \in C_w([0,T];L^2(\Omega;\R{m})$, and there is a subsequence $\{u^{j_k}\}$ such that
\[\int_\Omega u^{j_k}(t,x)\cdot w(x){\dd}x \rightarrow\int_\Omega u(t,x)\cdot w(x){\dd}x\]
uniformly in $C([0,T])$ for all $w\in L^2(\Omega;\R{m})$ as $k \rightarrow \infty$.
\end{lemma}

\begin{proof} We shall apply the Arzel\`a--Ascoli theorem for sequences of uniformly bounded and equicontinuous functions. For $w \in L^2(\Omega;\R{m})\cap W_0^{1,s}(\Omega;\R{m})$, we let $f^j:[0,T]\rightarrow\R{}$ be the family of functions defined by
\[f^j(t)=\int_\Omega u^j(t,x)\cdot w(x){\dd}x.\]
Since $\partial_tu^j \in L^{s'}\!(0,T;W^{-1,s'}\!(\Omega;\R{m}))$, it automatically follows that $u^j \in C([0,T];W^{-1,s'}\!(\Omega;\R{m}))$, whereby also $u^j \in C_w([0,T];W^{-1,s'}\!(\Omega;\R{m}))$. As $u^j \in  L^\infty(0,T; L^2(\Omega;\R{m}))$, it then follows from Lemma \ref{lem-Cw-Cw*} with $X = L^2(\Omega;\R{m})$ and $Y = L^2(\Omega;\R{m}) + W^{-1,s'}\!(\Omega;\R{m}) = (L^2(\Omega;\R{m}) \cap W^{1,s}_0(\Omega;\R{m}))'$, that $u^j \in  C_w([0,T]; L^2(\Omega;\R{m}))$. We note here that the duality $(L^2(\Omega;\R{m}) \cap W^{1,s}_0(\Omega;\R{m}))' = L^2(\Omega;\R{m}) + W^{-1,s'}\!(\Omega;\R{m})$ is a consequence of the Duality Theorem (cf. Theorem 2.7.1 on p.32 in Bergh \& L\"ofstr\"om \cite{BL}), because $L^2(\Omega;\R{m}) \cap W^{1,s}_0(\Omega;\R{m})$ is dense in both $L^2(\Omega;\R{m})$ and $W^{1,s}_0(\Omega;\R{m})$ (e.g., because $C^\infty_0(\Omega;\R{m})$ is dense in both $L^2(\Omega;\R{m})$ and $W^{1,s}_0(\Omega;\R{m})$). Hence, $f^j \in C([0,T])$.  We further have that
\begin{align*}
|f^j(t)|\leq \int_\Omega |u^j(t,x)||w(x)|{\dd}x
\leq \|u^j\|_{L^\infty(0,T;L^2(\Omega;\R{m}))}\|w\|_{L^2(\Omega;\R{m})}
\leq C\|w\|_{L^2(\Omega;\R{m})},
\end{align*}
which gives the boundedness of the sequence $\{f^j\}$ in $C([0,T])$ for each $w\in L^2(\Omega;\R{m})\cap W_0^{1,s}(\Omega;\R{m})$. In order to verify equicontinuity of the sequence, note that, for all $w \in L^2(\Omega;\R{m})\cap W_0^{1,s}(\Omega;\R{m})$,
\begin{align*}
|f^j(t+h)-f^j(t)|=\left|\int_\Omega(u^j(t+h,x)-u^j(t,x))\cdot w(x){\dd}x\right|
=\left|\int_t^{t+h}\int_\Omega\partial_tu^j(s,x)\cdot w(x){\dd}x{\dd}s\right|\\
\leq \|\partial_tu^j\|_{L^{s'}\!(0,T;W^{-1,s'}\!(\Omega;\R{m}))}|h|^{\frac{1}{s}}\|w\|_{W^{1,s}(\Omega;\R{m})}
\leq C|h|^{\frac{1}{s}}\|w\|_{W^{1,s}(\Omega;\R{m})}.
\end{align*}
Thus by the Arzel\`a--Ascoli theorem there is a subsequence $\{f^{j_k}\}$, which converges uniformly as $k\rightarrow\infty$ to some function $f \in C([0,T])$. We note that as this convergence is uniform and the sequence $f^{j_k}$ is uniformly bounded we have that
\[\lim_{k\rightarrow\infty}\int_0^T\varphi(t) \int_\Omega u^{j_k}(t,x)\cdot w(x){\dd}x{\dd}t
=\lim_{k\rightarrow\infty}\int_0^T \varphi(t) f^{j_k}(t){\dd}t=\int_0^T \varphi(t) f(t){\dd}t\qquad \forall\, \varphi \in C^\infty([0,T]).\]
However, by the assumed weak convergence, we also have that
\[\lim_{k\rightarrow\infty}\int_0^T\varphi(t)\int_\Omega u^{j_k}(t,x)\cdot w(x){\dd}x{\dd}t=\int_0^T\varphi(t) \int_\Omega u(t,x)\cdot w(x){\dd}x{\dd}t \qquad \forall\, \varphi \in C^\infty([0,T]),\]
allowing us to identify
\[f(t)=\int_\Omega u(t,x)\cdot w(x){\dd}x \qquad \forall\,t \in [0,T].\]
Thus we have shown that, for all $w \in L^2(\Omega;\R{m})\cap W_0^{1,s}(\Omega;\R{m})$,
\[\int_\Omega u^{j_k}(t,x)\cdot w(x){\dd}x\rightarrow\int_\Omega u(t,x)\cdot w(x){\dd}x,\qquad \mbox{uniformly in $C([0,T])$}.\]
As $L^2(\Omega;\R{m})\cap W_0^{1,s}(\Omega;\R{m})$ is dense in $L^2(\Omega;\R{m})$, it then follows that, for all $w \in L^2(\Omega;\R{m})$,
\[\int_\Omega u^{j_k}(t,x)\cdot w(x){\dd}x\rightarrow\int_\Omega u(t,x)\cdot w(x){\dd}x,\qquad \mbox{uniformly in $C([0,T])$},\]
which in particular means that for all $t\in[0,T]$ we have $u^{j_k}(t,\cdot)\rightharpoonup u(t,\cdot)$ in the space $L^2(\Omega;\R{m})$.
\end{proof}
Finally, we recall the following standard Lebesgue space interpolation result, whose proof is omitted.
\begin{lemma}\label{parabolicinterpolation}
Let $u:[0,T]\times\Omega\rightarrow\R{m}$, where $\Omega\subset\R{n}$ is open and bounded, satisfy
\[u\in L^p(0,T;W_0^{1,p}(\Omega;\R{m}))\cap L^\infty(0,T;L^2(\Omega;\R{m}))\]
for some $p\geq1$. Then,
\[u\in L^{\frac{p(n+2)}{n}}(Q_T;\R{m}).\]
\end{lemma}
\section{Spatial discretization and convergence of the semidiscrete problem}\label{sec:3}

In this section we prove the existence of solutions to the system \eqref{modelprob1}--\eqref{modelconstants}. In preparation for the construction of the fully discrete numerical approximation of the class of problems under consideration described in the next section, our proof of existence of large-data global-in-time Young measure solutions will be based on performing a spatial finite element approximation. We shall therefore suppose henceforth that $\Omega$ is a Lipschitz polytope in $\R{n}$.
The proof presented below can be replicated on any bounded Lipschitz domain $\Omega$ by replacing the finite element basis, consisting of continuous piecewise linear basis functions satisfying a homogeneous Dirichlet boundary condition on $\partial \Omega$ by an abstract Galerkin basis. Suppose further that $0<h_0\ll \mbox{diam}(\Omega)$ and that $\{\mathcal{T}_h\}_{h \in (0,h_0]}$ is a shape-regular family of subdivisions $\mathcal{T}_h$ of $\overline\Omega$ into closed simplexes $\Delta$, where $h=\max_{\Delta \in \mathcal{T}_h} \mbox{diam$(\Delta)$}$. Consider the finite element space
\[V^h:=\{v^h\in W^{1,\infty}_0(\overline{\Omega}) : v^h|_{\Delta}\ \text{is affine for all}\ \Delta \in \mathcal{T}_h\}.\]
By $V^h_m$ we denote the space of $m$-component vector-valued functions, each of whose components lies in $V^h$. We shall assume that $\{\mathcal{T}_h\}_{h \in (0,h_0]}$ is such that the $L^2(\Omega;\R{m})$ orthogonal projector $P^h$ onto $V^h_m$ is stable in $W^{1,\hat{q}}_0(\Omega;\R{m})$; i.e., there exists a positive constant $C$, independent of $h \in (0,h_0]$, such that
\begin{align}\label{e:stab}
\|D(P^h\varphi) \|_{L^{\hat{q}}(\Omega;\R{m\times n})} \leq C \|D \varphi\|_{L^{\hat{q}}(\Omega;\R{m \times n})}\qquad \forall\, \varphi \in W^{1,\hat{q}}_0(\Omega;\R{m}), \; \mbox{where $\hat{q}:= \max\{q,2\}$}.
\end{align}
On globally quasiuniform subdivisions \eqref{e:stab} is a consequence of a global inverse inequality and inequality (7) in \cite{DDW}; we note, however, that the stability inequality \eqref{e:stab} is in fact valid under less restrictive assumptions on the subdivision than quasiuniformity (see, for example, \cite{CT}).

Let $\{\phi^h_i\}_{i=1}^{\mathcal{N}(h)}$ be a basis for $V^h_m$. We shall therefore seek an approximate solution $u_h \in V^h_m$ in the form
\[u^h(t,x)=\sum_{i=1}^{\mathcal{N}(h)}\alpha^h_i(t)\phi^h_i(x),\qquad t \in [0,T], \quad x \in \overline\Omega,\]
satisfying
\begin{subequations}
\begin{align}\label{Galeq}
( \partial_tu^h(t),v^h) + ( a(Du^h(t)),Dv^h)+ (Bu^h,v^h)= ( F,v^h)
\qquad \forall\, v^h \in V^h_m,
\end{align}
for all $t\in[0,T]$, and
\begin{align}\label{Galeqini}
u^h(0)=u^h_0,
\end{align}
\end{subequations}
with $u^h_0\in V^h_m$ and $u^h_{0}\rightarrow u_0$ (strongly) in $L^2(\Omega;\R{m})$ as $h \rightarrow 0_+$. The system given by \eqref{Galeq} is a system of ODEs for the coefficients $\alpha^h_i$, $i=1,\dots, \mathcal{N}(h)$, and a solution to \eqref{Galeq}, \eqref{Galeqini} exists locally by Peano's Theorem on some interval $[0,T_h) \subset [0,T]$, thanks to the assumed (local Lipschitz) continuity of $a$. Our existence theorem below is based on proving that $u^h$ can be extended to
the final time $T>0$ for each $h \in (0,h_0]$, and that as $h \rightarrow 0_+$ a subsequence of the sequence of approximate solutions converges, in a sense to be made precise, to a Young measure solution of the problem.

\begin{theorem}\label{YMS}
Suppose that $\Omega\subset\R{n}$ is a Lipschitz polytope, let $u_0\in L^2(\Omega;\R{m})$, assume that
\[ p:=\min_{i=1,\ldots, m}p_i>\max\bigg\{\frac{2n}{n+2},1\bigg\}\quad \mbox{and}\quad q:=\max_{i=1,\ldots,m}p_i\]
satisfy $q-p<1$, let $F\in L^{p'}\!(Q_T;\R{m})\cap L^{2}(Q_T;\R{m})$, and suppose that $a:\R{m \times n}\rightarrow\R{m \times n}$ is a locally Lipschitz mapping satisfying the assumptions \eqref{modelstructure}--\eqref{modelconstants}. Then, there exists a Young measure solution $(u,\nu)$ of the problem \eqref{modelprob1} with data given by \eqref{modelinitdata} and \eqref{modelboundarydata}. Furthermore, there exists a subsequence (not indicated) of solutions $\{u^h\}$ to the semidiscrete problem \eqref{Galeq}, \eqref{Galeqini} such that
\[u^h\overset{\ast}{\rightharpoonup}u\ \text{in}\ L^\infty(0,T;L^2(\Omega;\R{m})),\qquad
Du^h_i \rightharpoonup Du_i \ \text{in}\ L^{p_i}(0,T; L^{p_i}(\Omega;\R{m\times n})),\quad i=1,\dots,m.
\]
\[a(Du^h)\rightharpoonup \langle \nu,a\rangle \ \text{in}\ L^{\hat{q}'}\!(Q_T;\R{m \times n}),
\qquad \partial_tu^h\rightharpoonup\partial_tu\ \text{in}\ L^{{\hat{q}}'}\!(0,T;W^{-1,{\hat{q}}'}\!(\Omega;\R{m})),
\quad\mbox{where $\hat{q}:=\max\{q,2\}$},\]
\[u^h \rightarrow u \ \text{in} \ C_w([0,T];L^2(\Omega;\R{m})).\]
\end{theorem}

\begin{proof}
As was noted above, \eqref{Galeq}, \eqref{Galeqini} is an initial-value problem for a system of ODEs for the coefficients $\alpha^h_i$, and a solution exists locally by Peano's Theorem on some interval $[0,T_h) \subset [0,T]$, thanks to the assumed (local Lipschitz) continuity of $a$. We begin by showing that we can extend the numerical solution $u^h \in V^h_m$, defined on $[0,T_h) \times \overline\Omega$ up to time $T$ for all $h \in (0,h_0]$, so that it is defined on the whole of $[0,T]\times \overline\Omega$. To this end, we take $v^h = u^h$ in \eqref{Galeq}; hence,
\begin{equation}\label{modelpreenergy}
\int_0^{\tau}\int_\Omega\partial_tu^h\cdot u^h+K(Du^h)|Du^h|^2
+Bu^h\cdot u^h{\dd}x{\dd}t=\int_0^{\tau}\int_\Omega F\cdot u^h{\dd}x{\dd}t\qquad \forall\, \tau \in (0,T_h),
\end{equation}
which, by using \eqref{b-positive} and \eqref{modelgrowth} on the left-hand side and H\"older's inequality on the right-hand side, yields
\begin{align}\label{modelenergy0}
\begin{split}
&\|u^h(\tau,\cdot)\|_{L^2(\Omega;\R{m})}^2+\int_0^{\tau}\int_\Omega
\left[\sum_{i=1}^m(\mu_i^2+|Du_i^h|^2)^\frac{p_i-2}{2}\right]|Du^h|^2{\dd}x{\dd}t\\
&\qquad \leq c \left[\|u_0^h\|_{L^2(\Omega;\R{m})}^2
+ \left(\int_0^{T}\|F\|^{p'}_{L^{p'}\!(\Omega;\R{m})}  {\dd}t\right)^{\frac{1}{p'}} \left(\int_0^{\tau}\|u^h\|_{L^{p}(\Omega;\R{m})}^p {\dd}t\right)^{\frac{1}{p}}
\right]\qquad \forall\, \tau \in (0,T_h).
\end{split}
\end{align}
In order to bound the second term on the left-hand side of \eqref{modelenergy0} from below, we note that
\[ |Du_i^h|^{p_i} \leq (\mu_i^2+|Du_i^h|^2)^\frac{p_i-2}{2} |Du^h|^2\qquad \mbox{if $p_i \geq 2$},\]
and the following inequality holds for all $s \in [0,\infty)$ if $1<p_i <2$  and $\mu_i \neq 0$:
\[ s^{p_i} \leq
\left\{
\begin{array}{ll}
    2^{1-\frac{p_i}{2}}(\mu^2_i + s^2)^{\frac{p_i -2}{2}}s^2 & \mbox{for $s \geq |\mu_i|$}, \\
    (\mu_i^2 + s^2)^{\frac{p_i -2}{2}}s^2  + |\mu_i|^{p_i} & \mbox{for $0 \leq s \leq |\mu_i|$}.
\end{array}
 \right.\]
Using these, we deduce that
\begin{align*}
&\|u^h(\tau,\cdot)\|_{L^2(\Omega;\R{m})}^2+\int_0^{\tau}\int_\Omega \sum_{i=1}^m
|Du_i^h|^{p_i}{\dd}x{\dd}t\\
&\qquad \leq c \left[1 + \|u_0^h\|_{L^2(\Omega;\R{m})}^2
+ \left(\int_0^{T}\|F\|^{p'}_{L^{p'}\!(\Omega;\R{m})}  {\dd}t\right)^{\frac{1}{p'}} \left(\int_0^{\tau}\|u^h\|_{L^{p}(\Omega;\R{m})}^p {\dd}t\right)^{\frac{1}{p}}
\right]\qquad \forall\, \tau \in (0,T_h).
\end{align*}
Thus, using Young's inequality and Poincar\'e's inequality in $W^{1,p}_0(\Omega;\R{m})$ in order to absorb the final
factor on the right-hand side into the second term on the left-hand side, we arrive at the following energy inequality:
\begin{align}\label{modelenergy}
\sup_{t\in[0,T_h)}\|u^h(t,\cdot)\|_{L^2(\Omega;\R{m})}^2+\int_0^{T_h}\int_\Omega
\sum_{i=1}^m|Du_i^h|^{p_i}{\rm d}x{\dd}t\leq c\big(1+\|u_0^h\|_{L^2(\Omega;\R{m})}^2+\|F\|_{L^{p'}\!(Q_{T};\R{m})}^{p'}\big),
\end{align}
which provides uniform bounds on the sequence $\{u^h\}$, thanks to the strong convergence of $u_0^h\rightarrow u_0$ in $L^2(\Omega;\R{m})$ as $h \rightarrow 0_+$, which implies in particular that $\|u_0^h\|_{L^2(\Omega;\R{m})}$ is uniformly bounded in $h$. Hence, the right-hand side of \eqref{modelenergy} is bounded, independent of $h \in (0,h_0]$. This means that $T_h$ cannot be the maximal existence time, and $u^h$ can be therefore extended from $[0,T_h)$ beyond $T_h$ to the whole of the time interval $[0,T]$ for all $h \in (0,h_0]$. Hence,
\begin{align}\label{modelenergy-a}
\begin{split}
\sup_{t\in[0,T)}\|u^h(t,\cdot)\|_{L^2(\Omega;\R{m})}^2+\int_0^{T}\int_\Omega
\left[\sum_{i=1}^m(\mu_i^2+|Du_i^h|^2)^\frac{p_i-2}{2}\right]|Du^h|^2{\dd}x{\dd}t\\
\leq c\big(1 + \|u_0\|_{L^2(\Omega;\R{m})}^2 +\|F\|_{L^{p'}\!(Q_{T};\R{m})}^{p'}\big)
\end{split}
\end{align}
and also
\begin{align}\label{modelenergy-b}
\sup_{t\in[0,T)}\|u^h(t,\cdot)\|_{L^2(\Omega;\R{m})}^2+\int_0^{T}\int_\Omega \sum_{i=1}^m|Du_i^h|^{p_i}{\rm d}x{\dd}t\leq c\big(1+\|u_0\|_{L^2(\Omega;\R{m})}^2+\|F\|_{L^{p'}\!(Q_{T};\R{m})}^{p'}\big).
\end{align}
The uniform bound \eqref{modelenergy-b} implies the existence of a function
\[ u \in L^\infty(0,T;L^2(\Omega;\R{m}))\cap L^p(0,T;W^{1,p}_0(\Omega;\R{m})) \quad\mbox{with}\quad Du_i \in L^{p_i}(0,T;L^{p_i}(\Omega;\R{m\times n})),\quad i=1,\dots,m,\]
for which, up to a subsequence, we have
\begin{equation}\label{modelweakconvergence1}
u^h\overset{\ast}{\rightharpoonup}u\ \text{in}\ L^\infty(0,T;L^2(\Omega;\R{m}) )
\end{equation}
and
\begin{equation}\label{modelweakconvergence2}
Du^h_i \rightharpoonup Du_i \ \text{in}\ L^{p_i}(0,T; L^{p_i}(\Omega;\R{m\times n})),\quad i=1,\dots,m.
\end{equation}
Furthermore, thanks to \eqref{modelgrowth} and because the Euclidean norm of the vector $(|Du_1^h|,\ldots,|Du_m^h|)^{\rm T}$ is bounded by its 1-norm, we have the following bound:
\begin{align}\label{a-bound}
|a(Du^h)|&\leq c_1 \sum_{i=1}^m(\mu_i^2+|Du_i^h|^2)^\frac{p_i-2}{2}\left(\sum_{k=1}^m|Du_k^h|^2\right)^{\frac{1}{2}}
\leq c_1 \sum_{i=1}^m\sum_{k=1}^m(\mu_i^2+|Du_i^h|^2)^\frac{p_i-2}{2}|Du_k^h|.
\end{align}
Now we seek to bound each of the terms appearing in the last sum in \eqref{a-bound}. Clearly for the case $i=k$ we have an expression of the form $(\mu_i^2+|Du_i^h|^2)^\frac{p_i-1}{2}$ appearing, which is bounded in $L^{p_i'}\!(Q_T)$, uniformly with respect to $h \in (0,h_0]$, by the energy estimate \eqref{modelenergy-b}. In the case $i\neq k$ we proceed by using H\"older's inequality together with \eqref{modelenergy-a} and \eqref{modelenergy-b}, our objective being to bound all of these terms in $L^{p_i'}\!(Q_T)$ if $p_i\geq 2$ or in $L^{2}(Q_T)$ if $1<p_i<2$. We begin with the terms for which $p_i\geq2$:
\begin{align*}
&\int_0^T\int_\Omega \bigg[(\mu_i^2+|Du_i^h|^2)^\frac{p_i-2}{2}|Du_k^h|\bigg]^{p_i'}{\dd}x{\dd}t
=\int_0^T\int_\Omega (\mu_i^2+|Du_i^h|^2)^{\frac{p_i'(p_i-2)}{4}}|Du_k^h|^{p_i'}
\,(\mu_i^2+|Du_i^h|^2)^{\frac{p_i'(p_i-2)}{4}}{\dd}x{\dd}t\\
&\qquad \leq \left(\int_0^T\int_\Omega(\mu_i^2+|Du_i^h|^2)^\frac{p_i-2}{2}|Du_k^h|^2{\dd}x\  {\rm d}t\right)^\frac{p_i'}{2}\,\left(\int_0^T\int_\Omega(\mu_i^2+|Du_i^h|^2)^\frac{p_i}{2}{\dd}x\  {\rm d}t\right)^{1-\frac{p_i'}{2}}\\
&\qquad \leq c\left(1+\|u_0\|_{L^2(\Omega;\R{m})}^2+\|F\|_{L^{p'}\!(Q_T;\R{m})}^{p'}\right)^{\frac{p_i'}{2}} \left(1 + \|u_0\|_{L^{2}(\Omega;\R{m})}^2
+ \|F\|^{p'}_{L^{p'}\!(Q_T;\R{m})}\right)^{1-\frac{p_i'}{2}}.
\end{align*}
Now for the terms in which $1<p_i<2$, thanks to \eqref{modelenergy-a}, we have
\begin{align*}
\int_0^T\int_\Omega& \bigg[(\mu_i^2+|Du_i^h|^2)^\frac{p_i-2}{2}|Du_k^h|\bigg]^{2}{\dd}x{\dd}t\\
&=\int_0^T\int_\Omega (\mu_i^2+|Du_i^h|^2)^\frac{p_i-2}{2}(\mu_i^2+|Du_i^h|^2)^\frac{p_i-2}{2}|Du_k^h|^2{\dd}x{\dd}t\\
&\leq|\mu_i|^{p_i-2} \int_0^T\int_\Omega (\mu_i^2+|Du_i^h|^2)^{\frac{p_i-2}{2}}|Du_k^h|^{2}{\dd}x{\dd}t\\
&\leq c|\mu_i|^{p_i-2}(1+\|u_0\|_{L^2(\Omega;\R{m})}^2+\|F\|_{L^{p'}\!(Q_T;\R{m})}^{p'}).
\end{align*}
Thus we have bounded each term in the sum appearing on the right-hand side of \eqref{a-bound} in $L^{p_i'}\!(Q_T)$ (if $p_i\geq2$) or in $L^{2}(Q_T)$ (if $1<p_i<2$), uniformly with respect to $h \in (0,h_0]$, and therefore $\{a(Du^h)\}_{0<h\leq h_0}$ is a bounded sequence in $L^{{\hat{q}}'}\!(Q_T;\R{m \times n})$. More precisely, there exists a positive constant $C$, independent of $h$, such that
\begin{align}\label{e:au-bound}
\|a(Du^h)\|_{L^{{\hat{q}}'}\!(Q_T;\R{m \times n})} \leq C\qquad \forall\, h \in (0,h_0],
\end{align}
where we recall that $\hat{q}=\max\{q,2\}$, whereby $\hat{q}'= \min\{q',2\}$. Therefore, there is a subsequence (still indexed only by $h$) and a $\chi\in L^{{\hat{q}}'}\!(Q_T;\R{m \times n})$ such that
\begin{align}\label{a-limit}
 a(Du^h)\rightharpoonup\chi\qquad \mbox{in $L^{{\hat{q}}'}\!(Q_T;\R{m \times n})$}.
\end{align}

Using similar estimates to those above we can also show that there exists a positive constant $C$, independent of $h$, such that
\begin{equation}
\label{modeltimederivative}
\|\partial_tu^h\|_{L^{{\hat{q}}'}\!(0,T; W^{-1,{\hat{q}}'}\!(\Omega;\R{m}))} \leq C\qquad \forall\, h \in (0,h_0].
\end{equation}
Indeed, for each $\varphi \in L^{\hat{q}}(0,T;W^{1,{\hat{q}}}_0(\Omega;\R{m}))$, by \eqref{Galeq}, \eqref{e:stab} and \eqref{e:au-bound}, we have that
\begin{align*}
&\int_0^T (\partial_tu^h, \varphi ) \dd t = \int_0^T ( \partial_tu^h, P^h\varphi ) \dd t = - \int_0^T ( a(Du^h(t)),DP^h\varphi) \dd t + \int_0^T ( F - Bu^h ,P^h\varphi ) \dd t\\
&\qquad\leq C\|a(Du^h)\|_{L^{{\hat{q}}'}\!(Q_T;\R{m \times n})} \|D\varphi\|_{L^{\hat{q}}(Q_T;\R{m\times n})} + \bigg(\|F\|_{L^2(Q_T;\R{m})} + |B| \|u^h\|_{L^2(Q_T;\R{m})}\bigg) \|\varphi\|_{L^2(Q_T;\R{m})}\\
&\qquad\leq C (\|D\varphi\|_{L^{\hat{q}}(Q_T;\R{m\times n})} + \|\varphi\|_{L^{\hat{q}}(Q_T;\R{m})})  \leq  C \|\varphi\|_{L^{\hat{q}}(0,T;W^{1,{\hat{q}}}(\Omega;\R{m}))},
\end{align*}
and then \eqref{modeltimederivative} follows by noting that the dual space of $L^{\hat{q}}(0,T;W^{1,{\hat{q}}}_0(\Omega;\R{m}))$ is $L^{\hat{q}'}\!(0,T;W^{-1,{\hat{q}}'}\!(\Omega;\R{m}))$.

Now let $\psi\in L^{\hat{q}}(0,T; W_0^{1,{\hat{q}}}(\Omega;\R{m}))$ be the strong limit, as $h \rightarrow 0_+$ and $l \rightarrow \infty$, of the sequence of functions
\[\psi^{\mathcal{N}(h),l}(t,x)=\sum_{i=1}^{\mathcal{N}(h)}\beta^l_i(t)\phi^h_i(x)\]
in the $L^{\hat{q}}(0,T; W_0^{1,{\hat{q}}}(\Omega;\R{m}))$ norm, with $\beta^l_i\in C([0,T])$ and $\phi^h_i\in V^h_m$, $i=1,\dots, \mathcal{N}(h)$, $l=1,2,\dots$; such a limit necessarily exists by density. We then have that
\begin{align*}
&\int_0^T\int_{\Omega}\left[\partial_tu^h\cdot \psi+a(Du^h) : D\psi+Bu^h\cdot\psi-F\cdot\psi\right]{\dd}x{\dd}t\\
&= \int_0^T\int_{\Omega}\partial_tu^h\cdot \psi^{\mathcal{N}(h),l}+a(Du^h) :  D\psi^{\mathcal{N}(h),l}+Bu^h\cdot\psi^{\mathcal{N}(h),l}-F\cdot\psi^{\mathcal{N}(h),l}{\dd}x{\dd}t\\
&\ \ \ \ +\int_0^T\int_{\Omega}\bigg[\partial_tu^h\cdot (\psi-\psi^{\mathcal{N}(h),l})+a(Du^h) : (D\psi-D\psi^{\mathcal{N}(h),l})
+Bu^h\cdot(\psi-\psi^{\mathcal{N}(h),l})-F\cdot(\psi-\psi^{\mathcal{N}(h),l})\bigg]{\dd}x{\dd}t\\
&\leq \|\partial_tu^h\|_{L^{{\hat{q}}'}\!(0,T;W^{-1,{\hat{q}}'}\!(\Omega;\R{m}))}
\|\psi-\psi^{\mathcal{N}(h),l}\|_{L^{\hat{q}}(0,T;W_0^{1,{\hat{q}}}(\Omega;\R{m}))}\\
&\ \ \ \ +\|a(Du^h)\|_{L^{{\hat{q}}'}\!(0,T;L^{{\hat{q}}'}\!(\Omega;\R{m \times n}))}\|D\psi-D\psi^{\mathcal{N}(h),l}\|_{L^{{\hat{q}}}(0,T;L^{{\hat{q}}}(\Omega;\R{m \times n}))}\\
&\ \ \ \ +\|F\|_{L^{p'}\!(Q_T;\R{m})}\|\psi-\psi^{\mathcal{N}(h),l}\|_{L^{\hat{q}}(0,T;W_0^{1,{\hat{q}}}(\Omega;\R{m}))}
+c\|u^h\|_{L^{{\hat{q}}'}\!(Q_T;\R{m})}
\|\psi-\psi^{\mathcal{N}(h),l}\|_{L^{\hat{q}}(0,T;W_0^{1,{\hat{q}}}(\Omega;\R{m}))}\\
&\leq c(\|u_0\|_{L^2(\Omega;\R{m})},\mu_i, T, |\Omega|,n,F)\,(\|\psi-\psi^{\mathcal{N}(h),l}\|_{L^{\hat{q}}(0,T;W_0^{1,{\hat{q}}}(\Omega;\R{m}))}
+\|D\psi-D\psi^{\mathcal{N}(h),l}\|_{L^{{\hat{q}}}(0,T;L^{{\hat{q}}}(\Omega;\R{m \times n}))}).
\end{align*}
The right-hand side of this inequality converges to zero in the limit as $h\rightarrow0_+$ (and therefore $\mathcal{N}(h)\rightarrow\infty$) and $l \rightarrow \infty$. Thus we have that
\begin{align}\label{limit-h}
\lim_{h \rightarrow 0_+}\int_0^T\int_{\Omega}\partial_tu^h\cdot \psi+a(Du^h) : D\psi + Bu^h\cdot\psi{\dd}x{\dd}t
=\int_0^T\int_\Omega F\cdot \psi{\dd}x{\dd}t,
\end{align}
for all $\psi\in L^{\hat{q}}(0,T; W_0^{1,{\hat{q}}}(\Omega;\R{m}))$.

The uniform in $h$ bounds established in \eqref{modeltimederivative} and \eqref{modelenergy-b} now allow passage to the limit in the first and third term on the left-hand side of \eqref{limit-h}: by the uniform in $h$ bound on $\|u^h\|_{L^\infty(0,T;L^2(\Omega;\R{m}))}$, there exists a subsequence (still indexed by $h$) which converges weakly-$\ast$ to $u$ in $L^\infty(0,T;L^2(\Omega;\R{m}))$. This enables us to also pass to the limit in the third term of the integrand on the left-hand side of \eqref{limit-h}. In addition, thanks to \eqref{modeltimederivative}, one can extract a further subsequence (which we continue to index by $h$) for which, by the uniqueness of the weak limit,
\begin{align}\label{timeconvw}
\partial_tu^h\rightharpoonup\partial_tu\quad \text{in}\quad L^{{\hat{q}}'}\!(0,T;W^{-1,{\hat{q}}'}\!(\Omega;\R{m})),
\end{align}
meaning that
\begin{equation*}
\lim_{h\rightarrow0_+}
\int_0^T (\partial_tu^h, \psi) {\dd}t=\int_{0}^T \langle \partial_tu,\psi\rangle{\dd}t,
\end{equation*}
for all $\psi\in L^{\hat{q}}(0,T;W^{1,{\hat{q}}}_0(\Omega;\R{m}))$,
where $\langle \cdot,\cdot \rangle$ is the duality pairing between
$W^{-1,{\hat{q}}'}\!(\Omega;\R{m})$ and $W^{1,{\hat{q}}}_0(\Omega;\R{m})$. Hence,
\begin{align}\label{limit-h1}
\int_0^T\langle \partial_tu , \psi \rangle \dd t +\lim_{h \rightarrow 0_+}\int_0^T\int_{\Omega}a(Du^h) : D\psi
+ \int_0^T\int_{\Omega} Bu\cdot\psi{\dd}x{\dd}t
=\int_0^T\int_\Omega F\cdot \psi{\dd}x{\dd}t,
\end{align}
for all $\psi\in L^{\hat{q}}(0,T;W^{1,{\hat{q}}}_0(\Omega;\R{m}))$. The final convergence result in the statement of the theorem, that $u^h \rightarrow u$ in $C_w([0,T];L^2(\Omega;\R{m}))$, is a direct consequence of \eqref{modelweakconvergence1}, \eqref{timeconvw} and Lemma \ref{ArzAscApp}.

In order to pass to the limit in the second term in \eqref{limit-h1}, we shall invoke Theorem \ref{FTYM}. To this end, we begin by verifying that the subsequence $\{Du^{h_k}\}$ satisfies the tightness condition from Theorem \ref{FTYM}; i.e., that
\begin{equation}\label{eq:tight}
\lim_{s\rightarrow\infty}\sup_{k}|\{(t,x)\in([0,T]\times\Omega)\cap Q_R(0):\ |Du^{h_k}(t,x)|\geq s\}|=0,
\end{equation}
for every $R>0$. In order to show this we set
\[A_{k,s,R}:=\{(t,x)\in([0,T]\times\Omega)\cap Q_R(0):\ |Du^{h_k}(t,x)|\geq s\},\]
and apply Chebyshev's inequality:
\begin{align*}|A_{k,s,R}|&\leq\frac{1}{s^p}\int_{A_{k,s,R}}|Du^{h_k}|^p{\dd}x\\
&\leq\frac{c(1+ ||u_0||_{L^2(\Omega;\R{m})}^2 + \|F\|^{p'}_{L^{p'}\!(Q_T;\R{m})})}{s^p},
\end{align*}
where we have used \eqref{modelenergy-b} along with the fact that for $r_1\geq r_2$, we have $|s|^{r_2}\leq 1+|s|^{r_1}$ to transition from the first line to the second. The resulting bound is independent of $k$ (and $h_k$) and $R$, and so by passing $s\rightarrow\infty$ we deduce that \eqref{eq:tight} holds.

By taking $\mathfrak{K}=\R{m}$ (see Remark \ref{rem:2} following the statement of Theorem \ref{FTYM}), we may apply Theorem \ref{FTYM}, to deduce the existence of a parameterized family of Young measures $(\nu_{t,x})_{(t,x) \in Q_T}$ and a further subsequence (which we still index by $h$) for which
\begin{align}\label{a-limit-a}
a(Du^h)\rightharpoonup \langle \nu,a\rangle \qquad \text{in}\ L^\alpha(Q_T;\R{m \times n}),
\end{align}
where $1<\alpha<\frac{p}{q-1}$ (see Remark \ref{rem:1}, combined with the fact that $a(A)\leq c(1+|A|^{q-1})$, which follows from \eqref{modelgrowth} by recalling  that $q:=\max_{i=1,\ldots,m} p_i$, whence $a \in E_{p/\alpha}$). By comparing
\eqref{a-limit} and \eqref{a-limit-a} and noting the uniqueness of weak limits there is then a subsequence of $u^h$ (still indexed by $h$) such that
\begin{align}\label{a-lmit-b}
a(Du^h)\rightharpoonup\langle \nu,a\rangle\qquad \text{in}\ L^{\hat{q}'}\!(Q_T;\R{m}).
\end{align}
Here we have used that $\hat{q}' = \min\{q',2\} = \min\{\frac{q}{q-1},2\} \geq \min\{\frac{p}{q-1},2\}\geq \min\{\alpha,2\}$, with $\alpha$ as above, and noted that \eqref{a-limit-a} implies that $a(Du^h)\rightharpoonup \langle \nu ,a\rangle$ in $L^{\min\{\alpha,2\}}(Q_T;\R{m \times n})$ because $\alpha \geq {\min\{\alpha,2\}}$.

This allows us to complete the passage to the limit and obtain the existence of a pair $(u,\nu)$ such that
\[ u \in L^\infty(0,T;L^2(\Omega;\R{m}))\cap L^p(0,T;W^{1,p}_0(\Omega;\R{m})),\qquad Du_i  \in L^{p_i}(0,T;L^{p_i}(\Omega;\R{m\times n})),\quad i=1,\dots,m,\]
with
\[\partial_tu\in L^{\hat{q}'}\!(0,T; W^{-1,\hat{q}'}\!(\Omega;\R{m})), \quad  u \in C_w([0,T];L^2(\Omega;\R{m}))\quad \mbox{and}\quad \langle\nu ,a\rangle \in L^{\hat{q}'}\!(Q_T;\R{m\times n})\]
satisfying
\[\int_0^T\langle\partial_tu,\phi\rangle+\int_\Omega\langle \nu_{t,x},a\rangle : D\phi + Bu\cdot\phi{\dd}x{\dd}t=\int_0^T\int_\Omega F\cdot\phi{\dd}x{\dd}t\]
for all $\phi\in L^{\hat{q}}(0,T; W_0^{1,\hat{q}}(\Omega;\R{m}))$.

It remains to be shown that the initial condition is satisfied. By Lemma \ref{ArzAscApp} it follows that we can extract a further subsequence $h_l$ along which we have that
\begin{equation}\label{weakcontconv}
\int_\Omega u^{h_l}(t,x)\cdot\varphi(x){\dd}x\rightarrow\int_\Omega u(t,x)\cdot\varphi(x){\dd}x,
\end{equation}
uniformly in $C([0,T])$, for all $\varphi\in L^2(\Omega;\R{m})$. Let us decompose
\begin{align*}
\int_\Omega &(u_0(x)-u(0,x))\cdot\varphi(x){\dd}x\\
&=\int_\Omega(u_0(x)-u^{h_l}(0,x))\cdot\varphi(x){\dd}x+\int_\Omega (u^{h_l}(0,x)-u^{h_l}(t,x))\cdot\varphi(x){\dd}x\\
&\ \ \ \ +\int_\Omega (u^{h_l}(t,x)-u(t,x))\cdot\varphi(x){\dd}x+\int_\Omega (u(t,x)-u(0,x))\cdot\varphi(x){\dd}x
=: I_l + II_l(t) + III_l(t) + IV(t).
\end{align*}
The term $I_l$ converges to zero as $h_l\rightarrow0_+$ by the assumed strong convergence in $L^2(\Omega;\R{m})$ of the discretized initial condition $u_0^{h_l}$ to $u_0$, and the term $III_l(t)$ converges to zero, uniformly in $C([0,T])$ as $h_l\rightarrow0_+$ by \eqref{weakcontconv}. Hence, for any $\varepsilon>0$, there exists an $h_l^\ast=h_l^\ast(\varepsilon,\varphi) \in (0,h_0]$, such that
\[ \left|\int_\Omega(u_0(x)-u^{h_l^\ast}(0,x))\cdot\varphi(x){\dd}x\right|<\frac{1}{4}\varepsilon\quad\mbox{and}\quad
 \max_{t \in [0,T]}\left|\int_\Omega (u^{h_l^\ast}(t,x)-u(t,x))\cdot\varphi(x){\dd}x\right| < \frac{1}{4}\varepsilon.
\]
Consequently, for any $\varepsilon>0$, there exists an $h_l^\ast \in (0,h_0]$, such that
\begin{align*}
&\left|\int_\Omega (u_0(x)-u(0,x))\cdot\varphi(x){\dd}x\right|\\
&\qquad < \frac{1}{2}\varepsilon +\left|\int_\Omega (u^{h_l^\ast}(0,x)-u^{h_l^\ast}(t,x))\cdot\varphi(x){\dd}x\right|
+\left|\int_\Omega (u(t,x)-u(0,x))\cdot\varphi(x){\dd}x\right|.
\end{align*}
With $h_l^\ast$ fixed, we now pass to the limit $t\rightarrow0_+$  in this inequality. The term $II_l(t)$ converges to zero as $t\rightarrow0_+$ because $u^{h_l^\ast}$, as a solution to the Galerkin equation \eqref{Galeq}, is continuous and satisfies the discretized initial condition. The term $IV(t)$ converges to as $t \rightarrow 0_+$ since by Lemma \ref{ArzAscApp} we have that $u\in C_w([0,T];L^2(\Omega;\R{m}))$, the space of weakly continuous functions from the interval $[0,T]$ into $L^2(\Omega;\R{m})$.
Thus, by passing to the limit $t \rightarrow 0_+$ in the last inequality we deduce that, for any $\varphi\in L^2(\Omega;\R{m})$ and any $\varepsilon>0$,
\begin{align*}
\left|\int_\Omega (u_0(x)-u(0,x))\cdot\varphi(x){\dd}x\right| < \frac{1}{2}\varepsilon.
\end{align*}
Taking $\varphi(x) = u_0(x)-u(0,x)$ and letting $\varepsilon \rightarrow 0_+$ then yields that $u_0(x)=u(0,x)$ for almost every $x\in\Omega$.
\end{proof}

%\begin{rem} We make explicit here the fact that each member of the family of Young measures, $(\nu_{t,x})$, is in fact an $m \times n$-matrix of Young measures, with each component being a measure on \R{m \times n}, and that the pairing $\langle \nu_{t,x},a\rangle$ is to be understood as an element of $\R{m \times n}$ for all $(t,x) \in Q_T$.
%\end{rem}
%

\section{The fully discrete scheme and its convergence analysis}\label{sec:4}

In the course of the proof of Theorem \ref{YMS} Peano's Theorem was applied in order to show that one can solve the Galerkin equation \eqref{Galeq}. We shall now employ an implicit Euler discretization scheme to numerically approximate the solution whose existence on $[0,T]$ is guaranteed by Peano's Theorem. Starting from our discretized initial condition $u_0^h\in V^h_m$, we inductively define $u_{i+1}^h$, $i=0,\dots, N-1$, as solutions to the following fully discrete problem.

Let $\Delta t = T/N$ where $N \geq 1$, and given $u_i^h\in V^h_m$, find $u_{i+1}^h \in V^h_m$ such that
\begin{equation}\label{timestep}
\int_\Omega \frac{u_{i+1}^h-u_i^h}{\Delta t}\cdot \phi^h+a(Du_{i+1}^h) : D\phi^h+Bu_{i+1}^h\cdot \phi^h{\dd}x=\int_\Omega F_{i+1}\cdot\phi^h{\dd}x\qquad \left\{\begin{array}{l}\mbox{for $i=0,\dots, N-1$},\\
\mbox{and for all $\phi^h\in V^h_m$}.
\end{array}\right.
\end{equation}
Here,
\[F_i(x):=\frac{1}{\Delta t}\int_{t_{i-1}}^{t_i}F(t,x){\dd}t,\]
and, letting $t_i:=i\Delta t$, we define on $Q_T$ the function
\[\tilde{F}_{\Delta t}(t,x):=F_i(x)\quad \text{for}\ t\in(t_{i-1},t_i],\quad i=1,\dots,N, \quad  x \in \Omega.\]

Given a function $u^h_i$, the fact that there exists a function $u^h_{i+1}$ satisfying \eqref{timestep} is standard, and can be shown using a simple consequence of Brouwer's Fixed Point Theorem (cf., for example, Corollary 1.1 on p.279 of \cite{GR}). We begin the convergence analysis with the following estimate, which is reminiscent of the energy estimate obtained in the proof of Theorem \ref{YMS}, for the semidiscrete problem considered there.

\begin{lemma}\label{discreteenergyestimate}
There exists a positive constant $c$, independent of $h$ and $\Delta t$,
such that the functions $u_i^h$, $i=1,\ldots,N$, satisfy, for all $h \in (0,h_0]$,
\begin{align}\label{uh-uniform}
\begin{aligned}
\max_{i=1,\ldots,N}\|u_i^h\|_{L^2(\Omega;\R{m})}^2+\sum_{i=1}^N\int_\Omega |u_i^h-u_{i-1}^h|^2{\dd}x + \Delta t\sum_{i=1}^N\int_\Omega a(Du_{i}^h) : Du_{i}^h{\dd}x\\
\leq c(\|u_0^h\|_{L^2(\Omega;\R{m})}^2+T\|F\|_{L^2(Q_T;\R{m})}^2).
\end{aligned}
\end{align}
We note here that $a(\xi) : \xi\geq 0$ for all $\xi \in \R{m\times n}$, and therefore each of the terms appearing on the left-hand side of the inequality \eqref{uh-uniform} is nonnegative. In addition, the expression on the right-hand side of \eqref{uh-uniform} is further bounded above by a constant that is independent of $h$ and $\Delta t$.
\end{lemma}
\begin{proof} We begin by choosing $\phi^h=u_{i+1}^h$ in \eqref{timestep}, and noting the following identity:
\[\frac{(a-b)\cdot a}{\Delta t}=\frac{|a|^2-|b|^2}{2\Delta t}+\frac{|a-b|^2}{2\Delta t}\qquad \forall\, a,b \in \R{m}.\]
Using this in \eqref{timestep} gives
\[\int_\Omega\frac{|u_{i+1}^h|^2-|u_i^h|^2}{2\Delta t}+\frac{|u_{i+1}^h-u_i^h|^2}{2\Delta t}+a(Du_{i+1}^h) : Du_{i+1}^h + Bu_{i+1}^h\cdot u_{i+1}^h{\dd}x=\int_\Omega F_{i+1}\cdot u_{i+1}^h{\dd}x.\]
Now we use that $Bv\cdot v\geq0$ for all $v \in \R{m}$, multiply by $2\Delta t$ and sum through $i=0,\ldots,M$ for some $M<N$ to get
\begin{align}\label{uh-bound1}
\begin{aligned}
\int_\Omega |u_{M+1}^h|^2{\dd}x&+\sum_{i=0}^M\int_\Omega|u_{i+1}^h-u_i^h|^2{\dd}x+2\Delta t\sum_{i=0}^M\int_\Omega a(Du_{i+1}^h) : Du_{i+1}^h{\dd}x\\
&\leq\int_\Omega |u_0^h|^2{\dd}x+2\sum_{i=0}^M\int_\Omega \left(\int_{t_i}^{t_{i+1}}F(t,x){\dd}t\right)\cdot u^h_{i+1}{\dd}x\\
&\leq\int_\Omega |u_0^h|^2{\dd}x+ 2T\|F\|_{L^2(Q_T;\R{m})}^2+ \frac{1}{2}\max_{i=1,\ldots,N}\|u_i^h\|_{L^2(\Omega;\R{m})}^2 .
\end{aligned}
\end{align}
By omitting the second and the third term from the left-hand side of the inequality \eqref{uh-bound1} noting the independence of its right-hand side of $M$, and then taking the maximum over $M=1,\dots,N-1$, yields the desired bound on the first term on the left-hand side of the inequality \eqref{uh-uniform}. We then return with that bound to the inequality \eqref{uh-bound1} to further estimate its right-hand side from above, whilst omitting the first term from the left-hand side of \eqref{uh-bound1};
thus we arrive at the desired bound on the second and third term on  the left-hand side of the inequality \eqref{uh-uniform}. The expression on the right-hand side of \eqref{uh-uniform} is clearly independent of $\Delta t$. The final assertion in the statement of the lemma is a consequence of the assumed strong convergence of $u_0^h$ to $u_0$ in $L^2(\Omega;\R{m})$.
\end{proof}
As the functions $u^h_i$, $i=0,\dots, N$, are defined only on $\{t_0=0,t_1, \dots, t_N=T\}\times \overline{\Omega}$ rather than on the whole of $\overline{Q_T}$, we shall next extend them to $\overline{Q_T}$. Thus we define the following two functions:
\[u_{\Delta t}^h(t):=\frac{t-t_{i-1}}{\Delta t}u_i^h+\frac{t_i-t}{\Delta t}u_{i-1}^h\ \ \ \text{for}\ t\in[t_{i-1},t_i],\qquad i=1,\dots,N,\]
which is continuous and piecewise linear with respect to $t \in [0,T]$, and
\[\tilde{u}_{\Delta t}^h(t):=u_i^h\ \ \ \text{for}\ t\in (t_{i-1},t_i],\quad i = 1,\dots, N,\qquad\text{and}\quad \tilde{u}_{\Delta t}^h(t):=u_0^h\ \ \ \text{for}\ t \in [-\Delta t,0], \]
which is piecewise constant with respect to $t \in [-\Delta t,T]$. With these new notations, \eqref{timestep} can be written as
\begin{equation}\label{timestep2}
\int_\Omega\partial_tu_{\Delta t}^h\cdot \phi^h+a(D\tilde{u}_{\Delta t}^h) : D\phi^h+B\tilde{u}_{\Delta t}^h\cdot\phi^h{\dd}x=\int_\Omega \tilde{F}_{\Delta t}\cdot\phi^h{\dd}x \quad \mbox{for $t \in [0,T]$, with $u^h_{\Delta t}(0,\cdot) = u^h_0$},
\end{equation}
for all  $\phi^h\in V^h_m$. We aim now to obtain uniform bounds on the sequences $\{u_{\Delta t}^h\}$ and $\{\tilde{u}_{\Delta t}^h\}$, which will allow us to complete our passage to the limit in $\Delta t$, possibly along some subsequence.

\begin{lemma}\label{discretebounds}
There exists a positive constant $c$, independent of $h$ and $\Delta t$, such that, for all $h \in (0,h_0]$,
\[ \|\tilde{u}^h_{\Delta t}\|^2_{L^\infty(0,T;L^2(\Omega;\R{m}))} \leq c (\|u_0^h\|_{L^2(\Omega;\R{m})}^2+\|F\|_{L^2(Q_T;\R{m})}^2),\]
\[ \|\tilde{u}^h_{\Delta t}-\tilde{u}^h_{\Delta t}(\cdot-\Delta t)\|^2_{L^2(0,T;L^2(\Omega;\R{m}))} \leq c (\|u_0^h\|_{L^2(\Omega;\R{m})}^2+\|F\|_{L^2(Q_T;\R{m})}^2),\]
\[\|a(D\tilde{u}^h_{\Delta t})\|_{L^{\hat{q}'}\!(Q_T;\R{m\times n})}^{\hat{q}'}\leq c\left(1+\|u_0^h\|_{L^2(\Omega;\R{m})}^2+\|F\|_{L^2(Q_T;\R{m})}^2 +\|F\|_{L^{p'}\!(Q_{T};\R{m})}^{p'}\right),\quad \mbox{where $\hat{q}:=\max\{q,2\}$},\]
\[
 \sum_{i=1}^m \|D(\tilde{u}^h_{\Delta t})_i\|_{L^{p_i}(0,T;L^{p_i}(\Omega;\R{n}))}^{p_i}\leq c\left(1+\|u_0^h\|_{L^2(\Omega;\R{m})}^2+\|F\|_{L^{p'}\!(Q_{T};\R{m})}^{p'}\right).
\]
Furthermore, the right-hand sides of these inequalities are bounded above by a positive constant, independent of $h \in (0,h_0]$ and
$\Delta t$.
\end{lemma}
\begin{proof}
The first two bounds follow directly from the bounds on the first two terms in \eqref{uh-uniform}, while the fourth inequality is proved by an argument that is completely analogous to the proof of the bound on the second term appearing on the left-hand side of \eqref{modelenergy-b}. We shall therefore focus our attention on showing the third bound in the statement of the lemma. To this end, we begin by noting that since $a(\xi):\xi \geq 0$ for all $\xi \in \R{m \times n}$, we have that
\[ |a(Du_i^h)|^{q'} = |a(Du_i^h) : a(Du_i^h)|^{\frac{q'}{2}} = \big|\big(a(Du_i^h) : Du_i^h \big)\, K(Du_i^h)\big|^{\frac{q'}{2}} = \left(a(Du_i^h) : D\tilde{u}^h_{\Delta t}\right)^{\frac{q'}{2}}  |K(Du_i^h)|^{\frac{q'}{2}}. \]

First suppose that $q>2$; then, $\hat{q}=q$ and $\hat{q}'=q' \in (1,2)$; hence, by Young's inequality,
\[ |a(Du_i^h)|^{\hat{q}'} = |a(Du_i^h)|^{q'} \leq \frac{q'}{2} a(Du_i^h) : Du_i^h + \frac{2-q'}{2} |K(Du_i^h)|^{\frac{q'}{2-q'}}.\]
Therefore, by \eqref{uh-uniform} and \eqref{modelgrowth}, we have that
\begin{align*}
&\int_0^T\int_\Omega |a(D\tilde{u}^h_{\Delta t})|^{\hat{q}'}{\dd}x{\dd}t = \int_0^T\int_\Omega |a(D\tilde{u}^h_{\Delta t})|^{q'}{\dd}x{\dd}t=\Delta t\sum_{i=1}^N\int_\Omega |a(Du_i^h)|^{q'}{\dd}x\\
&\qquad\leq \frac{q'}{2} \Delta t\sum_{i=1}^N\int_\Omega a(Du_i^h) : Du_i^h{\dd}x +\frac{2-q'}{2}\Delta t\sum_{i=1}^N\int_\Omega |K(Du_i^h)|^\frac{q'}{2-q'}{\dd}x\\
&\qquad \leq c(q)\left(\|u_0^h\|_{L^2(\Omega;\R{m})}^2+T \|F\|_{L^2(\Omega;\R{m})}^2\right) + c(q) \Delta t\sum_{i=1}^N\int_\Omega\left(c_1\sum_{j=1}^m\left(|\mu_j|^2 + |Du_{i,j}^h|^2\right)^{\frac{p_j-2}{2}}\right)^\frac{q'}{2-q'}{\dd}x\\
&\qquad \leq c(q) \left(\|u_0^h\|_{L^2(\Omega;\R{m})}^2+T\|F\|_{L^2(Q_T;\R{m})}^2\right)
+c(c_1,q,m)\Delta t\sum_{i=1}^N\int_\Omega\left[\sum_{j=1}^m |\mu_j|^{\frac{(p_j-2)q}{q-2}} + \sum_{\stackrel{\mbox{\footnotesize $j\!=\!1$}}{p_j \geq 2}}^m |Du_{i,j}^h|^{\frac{(p_j-2)q}{q-2}}\right] {\dd}x.
\end{align*}
As $q:=\max\{p_1,\ldots,p_m\} >2$, for any $j \in \{1,\dots,m\}$ such that $p_j \geq 2$ we have that  $\frac{(p_j-2)q}{q-2}<p_j$; it thus follows that
\begin{align*}
&\int_0^T\int_\Omega |a(D\tilde{u}^h_{\Delta t})|^{\hat{q}'}{\dd}x{\dd}t\leq c(q)\left(\|u_0^h\|_{L^2(\Omega;\R{m})}^2+T\|F\|_{L^2(Q_T;\R{m})}^2\right)\\
&\ \ \ \ \ \ +c(c_1, q,m)\,  T \, |\Omega| \sum_{j=1}^m |\mu_j|^{\frac{(p_j-2)q}{q-2}}
+ c(q,m,c_1,p_1,\dots,p_m,|\Omega|)\,\Delta t\sum_{i=1}^N \int_\Omega \left(\sum_{j=1}^m|Du_{i,j}^h|^{p_j}\right){\dd} x.
\end{align*}
The final term on the right-hand side of this inequality is then bounded analogously as the corresponding term in the proof of Theorem \ref{YMS} and we therefore omit the details (cf. \eqref{modelenergy-b}):
\begin{align*}
\Delta t\sum_{i=1}^N \int_\Omega \left(\sum_{j=1}^m|Du_{i,j}^h|^{p_j}{\rm d}x\right)\leq c\big(1+\|u_0^h\|_{L^2(\Omega;\R{m})}^2+\|F\|_{L^{p'}\!(Q_{T};\R{m})}^{p'}\big).
\end{align*}
Substituting this into the right-hand side of the previous inequality completes the proof of the lemma in the case when
$\hat{q} = q>2$.

If $\hat{q}=2$ (and therefore $\hat{q}'=2$), then $p_j \in (1,2]$ for all $j=1,\ldots,m$. It then follows from \eqref{modelgrowth} that
\[
0 \leq K(A) \leq c_1\sum_{i=1}^m |\mu_i|^{p_i-2} \qquad \forall\, A \in \mathbb{R}^{m \times n},
\]
whereby
\[ |a(Du_i^h)|^{\hat{q}'} = |a(Du_i^h)|^{2}= \left(a(Du_i^h) : D\tilde{u}^h_{\Delta t}\right)  K(Du_i^h) \leq
c_1\left(\sum_{i=1}^m |\mu_i|^{p_i-2}\right) \left(a(Du_i^h) : D\tilde{u}^h_{\Delta t}\right).\]
Hence, by Lemma \ref{discreteenergyestimate},
\begin{align*}
&\int_0^T\int_\Omega |a(D\tilde{u}^h_{\Delta t})|^{\hat{q}'}{\dd}x{\dd}t=\Delta t\sum_{i=1}^N\int_\Omega |a(Du_i^h)|^{q'}{\dd}x\\
&\qquad\leq c_1\left(\sum_{i=1}^m |\mu_i|^{p_i-2}\right) \Delta t\sum_{i=1}^N\int_\Omega a(Du_i^h) : Du_i^h{\dd}x\\
&\qquad \leq c(c_1,p_1,\ldots,p_m, \mu_1,\ldots, \mu_m, m, T)
\left(\|u_0^h\|_{L^2(\Omega;\R{m})}^2+\|F\|_{L^2(\Omega;\R{m})}^2\right).
\end{align*}
The final assertion in the statement of the lemma is a consequence of the assumed strong convergence of $u^h_0$ to $u_0$ as
$h \rightarrow 0_+$. That completes the proof of the lemma.
\end{proof}

Next, we prove the required uniform bounds on the sequence $\{u^h_{\Delta t}\}$.

\begin{lemma}\label{discretebounds1}
There exists a positive constant $c$, independent of $h$ and $\Delta t$, such that,  for all $h \in (0,h_0]$,
\[ \|u^h_{\Delta t}\|^2_{L^\infty(0,T;L^2(\Omega;\R{m}))} \leq c (\|u_0^h\|_{L^2(\Omega;\R{m})}^2+\|F\|_{L^2(Q_T;\R{m})}^2);\]
furthermore, with  $\hat{q}:=\max\{q,2\}$,
\[ \|\partial_t u^h_{\Delta t}\|_{L^{\hat{q}'}\!(0,T;W^{-1,\hat{q}'}\!(\Omega;\R{m}))} \leq c\]
and
\[(\Delta t)^{-1}\|\tilde{u}^h_{\Delta t}-\tilde{u}^h_{\Delta t}(\cdot-\Delta t)\|_{L^{\hat{q}'}\!(0,T;W^{-1,\hat{q}'}\!(\Omega;\R{m}))} \leq c.\]
\end{lemma}

\begin{proof}
The first inequality is just a restatement of the bound on the first term in the inequality \eqref{uh-uniform} in terms of
$u^h_{\Delta t}$, and the third inequality is a restatement of the second inequality in terms of $\tilde{u}^h_{\Delta t}$. We shall therefore concentrate on the proof of the second stated inequality. The argument proceeds along the same lines as the proof of \eqref{modeltimederivative}. For each $\varphi \in L^{\hat{q}}(0,T;W^{1,\hat{q}}_0(\Omega;\R{m}))$ we have, using the stability in $W^{1,\hat{q}}_0(\Omega;\R{m})$ of the $L^2(\Omega;\R{m})$ orthogonal projector $P^h:L^2(\Omega;\R{m})\rightarrow V^h_m$ (cf. \eqref{e:stab}), that
\begin{align*}
\int_0^T (\partial_t u^h_{\Delta t}, \varphi) \dd t &= \int_0^T (\partial_t u^h_{\Delta t}, P^h\varphi) \dd t
 = \sum_{i=0}^{N-1} \int_{t_i}^{t_{i+1}} \int_\Omega \frac{u_{i+1}^h - u_i^h}{\Delta t} \cdot P^h\varphi \dd x \dd t\\
&= -  \sum_{i=0}^{N-1} \int_{t_i}^{t_{i+1}} \int_\Omega a(Du^h_{i+1}) \cdot DP^h\varphi \dd x \dd t +
\sum_{i=0}^{N-1} \int_{t_i}^{t_{i+1}} ({F}_{i+1} - Bu_{i+1}^h)\cdot P^h\varphi \dd x \dd t\\
 &= -  \int_0^T ( a(D\tilde{u}^h_{\Delta t}(t)),DP^h\varphi) \dd t + \int_0^T ( \tilde{F}_{\Delta t} - B\tilde{u}^h_{\Delta t} ,P^h\varphi ) \dd t\\
&\leq \|a(D\tilde{u}^h_{\Delta t})\|_{L^{{\hat{q}}'}\!(Q_T;\R{m \times n})} \|D\varphi\|_{L^{\hat{q}}(Q_T;\R{m\times n})} \\
&\qquad + \bigg(\|\tilde{F}_{\Delta t}\|_{L^2(Q_T;\R{m})} + |B| \|\tilde{u}^h_{\Delta t}\|_{L^2(Q_T;\R{m})}\bigg) \|\varphi\|_{L^2(Q_T;\R{m})}\\
&\leq c (\|D\varphi\|_{L^{\hat{q}}(Q_T;\R{m\times n})} + \|\varphi\|_{L^{\hat{q}}(Q_T;\R{m})})  \leq  c \|\varphi\|_{L^{\hat{q}}(0,T;W^{1,{\hat{q}}}(\Omega;\R{m}))},
\end{align*}
where in the transition to the last line we have used the third and the first bound from Lemma \ref{discretebounds}, together with the assumed strong convergence of $u^h_0$ to $u_0$ in $L^2(\Omega;\R{m})$ to deduce that the constant $c$ is independent of $h$ and $\Delta t$. Dividing through by $\|\varphi\|_{L^{\hat{q}}(0,T;W^{1,{\hat{q}}}(\Omega;\R{m}))}$, taking the supremum over all $\varphi \in L^{\hat{q}}(0,T;W^{1,{\hat{q}}}_0(\Omega;\R{m}))$ and recalling the definition of the norm of the dual space $L^{\hat{q}'}\!(0,T;W^{-1,\hat{q}'}\!(\Omega;\R{m}))=(L^{\hat{q}}(0,T;W^{1,{\hat{q}}}_0(\Omega;\R{m})))'$ the stated inequality directly follows.
\end{proof}

We shall now discuss the convergence of the sequence of functions $\{u_{\Delta t}^h\}$ with $h \in (0,h_0]$ fixed.

\begin{lemma}\label{weaktimeconvergence}
There exists a subsequence of $\Delta t$, labelled $\Delta t_k$, and a function $u^h\in C([0,T];V^h_m)$ such that $u_{\Delta t_k}^h\rightarrow u^h$ in $C([0,T];L^2(\Omega;\R{m}))$, and $\partial_tu_{\Delta t_k}^h\rightharpoonup\partial_tu^h$ in $L^{\hat{q}'}\!(0,T;W^{-1,\hat{q}'}\!(\Omega;\R{m}))$ for any (fixed) $h \in (0,h_0]$, as $k \rightarrow \infty$.
\end{lemma}
\begin{proof}
We begin by noting that, for any $\Delta t$ and any $h \in (0,h_0]$, the functions $u^h_{\Delta t}$ are continuous in both $t$ and $x$. Thanks to the first two bounds in Lemma \ref{discretebounds1} there exists a subsequence of $\Delta t$, labelled $\Delta t_k$, and a function $u^h\in L^\infty(0,T;V^h_m)$ such that $u_{\Delta t_k}^h \overset{\ast}\rightharpoonup u^h$ in $L^\infty(0,T;L^2(\Omega;\R{m}))$, and $\partial_tu_{\Delta t_k}^h\rightharpoonup\partial_tu^h$ in $L^{\hat{q}'}\!(0,T;W^{-1,\hat{q}'}\!(\Omega;\R{m}))$. By Lemma \ref{ArzAscApp}, $u^h\in C_w([0,T];L^2(\Omega;\R{m}))$ and $u_{\Delta t_k}^h \rightarrow u^h$ in $C_w([0,T];L^2(\Omega;\R{m}))$. As both $u^h_{\Delta t}(t,\cdot)$ and $u^h(t,\cdot)$ belong to the finite-dimensional space $V^h_m$ for all $t \in [0,T]$, it follows that $u^h\in C([0,T];L^2(\Omega;\R{m}))$ and $u_{\Delta t_k}^h \rightarrow u^h$ in $u^h\in C([0,T];L^2(\Omega;\R{m}))$, because $u^h\in C_w([0,T];V^h_m) = C([0,T];V^h_m)$ and because a sequence converges in $C_w([0,T];V^h_m)$ if, and only if, it converges in $C([0,T];V^h_m)$ as $k \rightarrow \infty$ for any (fixed) $h \in (0,h_0]$.
\end{proof}

Next, we study the convergence of the sequence of functions $\{\tilde{u}_{\Delta t}^h\}$ with $h \in (0,h_0]$ fixed.

\begin{lemma}\label{nonlinearconvergence}
There exists a function $\tilde{u}^h \in L^\infty(0,T;V^h_m)$ and a subsequence $\Delta t_k$ such that
\[\tilde{u}_{\Delta t_k}^h \rightharpoonup \tilde{u}^h
 \quad \mbox{in} \ \ \bigtimes_{j=1}^m L^{p_j}(0,T;W^{1,p_j}_0(\Omega;\R{m}))
 \qquad\mbox{and}\qquad \tilde{u}_{\Delta t_k}^h \rightarrow \tilde{u}^h
 \quad \mbox{in} \ L^r(0,T;L^2(\Omega;\R{m})),
\]
for any $r \in [1,\infty)$ and fixed $h \in (0,h_0]$, as $k\rightarrow \infty$. Furthermore, along this subsequence we have $a(D\tilde{u}_{\Delta t_k}^h)\rightarrow a(D\tilde{u}^h)$ in $L^{s}(Q_T;\R{m\times n})$ for any $s \in [1,\hat{q}')$ and $a(D\tilde{u}_{\Delta t_k}^h)\rightharpoonup a(D\tilde{u}^h)$ in $L^{\hat{q}'}\!(Q_T;\R{m\times n})$, for any fixed $h \in (0,h_0]$, as $k\rightarrow \infty$.
\end{lemma}
\begin{proof} Thanks to the first and fourth bound from Lemma \ref{discretebounds}, there exists a
$\tilde{u}^h \in L^\infty(0,T;V^h_m)$ and a subsequence $\Delta t_k$ such that $\tilde{u}_{\Delta t_k}^h \overset{\ast}\rightharpoonup \tilde{u}^h$ in $L^\infty(0,T;L^2(\Omega;\R{m}))$ and
\[\tilde{u}_{\Delta t_k}^h \rightharpoonup \tilde{u}^h
 \qquad \mbox{in} \ \ \bigtimes_{j=1}^m L^{p_j}(0,T;W^{1,p_j}_0(\Omega;\R{m})),
\]
for any fixed $h \in (0,h_0]$, as $k\rightarrow \infty$.

By applying Theorem 1 from \cite{DJ} (with $X=W^{1,p}_0(\Omega;\R{m})$, $B=L^2(\Omega;\R{m})$ and $Y=W^{-1,\hat{q}'}\!(\Omega;\R{m})$ and $r=1$, there), we deduce that $\{\tilde{u}^h_{\Delta t_k}\}$ is relatively compact in $L^p(0,T;L^2(\Omega;\R{m}))$, for each fixed $h \in (0,h_0]$; therefore, because it is a bounded sequence in $L^\infty(0,T;L^2(\Omega;\R{m}))$, by interpolation  $\{\tilde{u}^h_{\Delta t_k}\}$ is also relatively compact in $L^r(0,T;L^2(\Omega;\R{m}))$, for each fixed $h \in (0,h_0]$ and all $r \in [1,\infty)$.

Hence, in particular, $\{\tilde{u}^h_{\Delta t_k}\}$ is relatively compact in $L^1(0,T;L^1(\Omega;\R{m}))$, for each fixed $h \in (0,h_0]$. As $\tilde{u}^h_{\Delta t_k}(t,\cdot) \in V^h_m$ for all $t \in [0,T]$, and $V^h_m$ is finite-dimensional, for each fixed $h \in (0,h_0]$, by norm equivalence in finite-dimensional spaces, it follows that $\{\tilde{u}^h_{\Delta t_k}\}$ is relatively compact in $L^1(0,T;W^{1,1}_0(\Omega;\R{m}))$, for each fixed $h \in (0,h_0]$. Therefore $D\tilde{u}^h_{\Delta t_k} \rightarrow D\tilde{u}^h$ in $L^1(Q_T;\R{m \times n})$ as $k \rightarrow \infty$, for each fixed $h \in (0,h_0]$. Thus we can extract a subsequence with respect to $k$ (not indicated) such that $D\tilde{u}^h_{\Delta t_k} \rightarrow D\tilde{u}^h$ a.e. on $Q_T$ as $k \rightarrow \infty$, for each fixed $h \in (0,h_0]$. Thanks to the continuity of $a$, we then have that $a(D\tilde{u}^h_{\Delta t_k}) \rightarrow a(D\tilde{u}^h)$ a.e. on $Q_T$ as $k \rightarrow \infty$, for each fixed $h \in (0,h_0]$. Since by the third inequality from Lemma \ref{discretebounds} the sequence $\{a(D\tilde{u}^h_{\Delta t_k})\}$ is weakly compact in $L^{\hat{q}'}\!(Q_T;\R{m \times n})$ and hence in particular also in $L^1(Q_T;\R{m \times n})$, for each fixed $h \in (0,h_0]$, it follows by Vitali's convergence theorem that $a(D\tilde{u}^h_{\Delta t_k}) \rightarrow a(D\tilde{u}^h)$ strongly in $L^1(Q_T;\R{m \times n})$ as $k \rightarrow \infty$, for each fixed $h \in (0,h_0]$, and therefore also strongly in  $L^s(Q_T;\R{m \times n})$ as $k \rightarrow \infty$, for each fixed $h \in (0,h_0]$ and any $s \in [1,\hat{q}')$. By the uniqueness of the weak limit it then follows that $a(D\tilde{u}_{\Delta t_k}^h)\rightharpoonup a(D\tilde{u}^h)$ weakly in $L^{\hat{q}'}\!(Q_T;\R{m\times n})$, for any fixed $h \in (0,h_0]$, as $k\rightarrow \infty$.
\end{proof}

We now show that the limit functions $u^h$ and $\tilde{u}^h$ defined above are equal when considered as elements of the space $L^2(Q_T;\R{m})$.

\begin{lemma} The limiting functions $u^h$ and $\tilde{u}^h$ are equal in the space $L^2(Q_T;\R{m})=L^2(0,T;L^2(\Omega;\R{m}))$.
\end{lemma}
\begin{proof} We apply the triangle inequality to see that, with $\Delta t_k = T/N_k$ and $h \in (0,h_0]$,
\begin{align*}\|u^h-\tilde{u}^h\|_{L^2(0,T; L^2(\Omega;\R{m}))}&\leq \|u^h-u_{\Delta t_{k}}^h\|_{L^2(0,T;L^2(\Omega;\R{m}))}+\|u_{\Delta t_{k}}^h-\tilde{u}_{\Delta t_{k}}^h\|_{L^2(0,T;L^2(\Omega;\R{m}))}\\
&\ \ \ \ \ \ +\|\tilde{u}_{\Delta t_{k}}^h-\tilde{u}^h\|_{L^2(0,T;L^2(\Omega;\R{m}))}.
\end{align*}
The first and last terms converge to zero in the limit as $\Delta t_{k}\rightarrow0_+$  thanks to the convergence results stated in Lemma \ref{weaktimeconvergence} and Lemma \ref{nonlinearconvergence} (with $r=2$), respectively, for each fixed $h \in (0,h_0]$. We now show that the second term also converges to zero in the limit:
\begin{align*}
\|u_{\Delta t_{k}}^h-\tilde{u}_{\Delta t_{k}}^h\|_{L^2(0,T;L^2(\Omega;\R{m}))}^2&=\int_0^T\int_\Omega|u_{\Delta t_{k}}^h-\tilde{u}_{\Delta t_{k}}^h|^2{\dd}x{\dd}t\\
&=\sum_{i=1}^{N_k}\int_\Omega \int_{t_{i-1}}^{t_i}\left|\frac{t-t_{i-1}-\Delta t_{k}}{\Delta t_{k}}u_i^h+\frac{t_i-t}{\Delta t_{k}}u_{i-1}^h\right|^2{\dd}t{\dd}x\\
&=\frac{\Delta t_{k}}{3}\sum_{i=1}^{N_k}\int_\Omega |u_i^h-u_{i-1}^h|^2{\dd}x\\
&=\frac{\Delta t_{k}}{3} \|\tilde{u}^h_{\Delta t_k}-\tilde{u}^h_{\Delta t_k}(\cdot-\Delta t)\|^2_{L^2(0,T;L^2(\Omega;\R{m}))}.
\end{align*}
Thanks to the second inequality in Lemma \ref{discretebounds}, the right-hand side of the last inequality converges to zero in the limit of $\Delta t_k \rightarrow 0_+$ ($N_k \rightarrow \infty$ as $k \rightarrow \infty$). Hence, $u^h=\tilde{u}^h$, as has been asserted.
\end{proof}

Next, we discuss the attainment of the initial condition.

\begin{lemma} The limiting function $u^h$ satisfies the initial condition in the following sense:
\[\lim_{t\rightarrow0_+}\|u^h(t,\cdot)-u_0^h(\cdot)\|_{L^2(\Omega;\R{m})}=0,\]
for each fixed $h \in (0,h_0]$.
\end{lemma}

\begin{proof}
Thanks to Lemma \ref{weaktimeconvergence}, for any $\epsilon>0$ there is an $\ell \in\mathbb{N}$ such that $\|u^h(t,\cdot)-u_{\Delta t_{k_l}}^h(t,\cdot)\|_{L^2(\Omega;\R{m})}<\epsilon$ for all $l\geq \ell$ and all $t \in [0,T]$. Furthermore, by the triangle inequality:
\[\|u^h(t,\cdot)-u_0^h(\cdot)\|_{L^2(\Omega;\R{m})}\leq \|u^h(t,\cdot)-u_{\Delta t_{k_l}}^h(t,\cdot)\|_{L^2(\Omega;\R{m})}+\|u_{\Delta t_{k_l}}^h(t,\cdot)-u_0^h(\cdot)\|_{L^2(\Omega;\R{m})}\]
for any $l \in \mathbb{N}$. Hence, in particular, for $l = \ell$ (fixed),
\[\|u^h(t,\cdot)-u_0^h(\cdot)\|_{L^2(\Omega;\R{m})}\leq \|u^h(t,\cdot)-u_{\Delta t_{k_\ell}}^h(t,\cdot)\|_{L^2(\Omega;\R{m})} + \epsilon.\]
As $u_{\Delta t_{k_\ell}}^h$ is a continuous function in $t$, passage to the limit in this inequality, using that $\lim_{t\rightarrow0_+}u_{\Delta t_{k_l}}^h(t,x)=u_0^h(x)$, implies that, for each $\epsilon>0$,
\[0 \leq \limsup_{t \rightarrow 0_+}\|u^h(t,\cdot)-u_0^h(\cdot)\|_{L^2(\Omega;\R{m})}\leq \epsilon.\]
Thus we have the asserted result.
\end{proof}

Finally, we focus on the convergence of the source term as the time step tends to $0$.

\begin{lemma}\label{Fconv} We have that $\lim_{\Delta t\rightarrow0_+}\|\tilde{F}_{\Delta t}(\cdot,x)-F(\cdot,x)\|_{L^r(0,T)}$ for all $r \in [1,\infty)$ and a.e. $x \in \Omega$.
\end{lemma}
\begin{proof}
The proof of this result is standard and is therefore omitted.
\end{proof}

We are now ready for passage to the limit in \eqref{timestep2} as $\Delta t \rightarrow 0_+$. By Lemma \ref{weaktimeconvergence}, Lemma \ref{nonlinearconvergence}, and Lemma \ref{Fconv}, we see that the limit function $u^h$ satisfies
\[\int_0^T\int_\Omega \partial_tu^h\cdot \Phi^h+a(Du^h) : D\Phi^h+Bu^h\cdot\Phi^h{\dd}x{\dd}t=\int_0^T\int_\Omega F\cdot\Phi^h{\dd}x{\dd}t \qquad \forall\, \Phi^h \in V^h_m, \quad \forall\,h \in (0,h_0],\]
with $u^h(0,x) = u^h_0(x)$ for $x \in \overline{\Omega}$. We can now pass to the limit $h \rightarrow 0_+$ by (the proof of) Theorem \ref{YMS}, to deduce that as the discretization parameters converge to zero the solution of the fully discrete scheme \eqref{timestep} converges to a Young measure solution of the problem \eqref{modelprob1}--\eqref{modelconstants} under consideration.

In preparation for the considerations in the next section, we now discuss continuous dependence of $u^h_{\Delta t}$ on the initial data. We start by considering two functions, $u_0^h \in V^h_m$ and $v_0^h\in V^h_m$, and we let $u_1^h$ and $v_1^h$ be respective solutions to \eqref{timestep} after a single time step. Subtracting the resulting equations satisfied by $u_1^h$ from the equation satisfied by $v_1^h$ and choosing as the test function $\phi^h=u_1^h-v_1^h$, it follows that
\[\int_\Omega |u_1^h-v_1^h|^2 +\Delta t(a(Du_1^h)-a(Dv_1^h))\cdot(Du_1^h-Dv_1^h){\dd}x \leq \int_\Omega(u_0^h-v_0^h)\cdot(u_1^h-v_1^h){\dd}x,\]
which after applying Young's inequality to the right-hand side and absorbing terms, gives
\[\|u_1^h-v_1^h\|_{L^2(\Omega;\R{m})}^2+2\Delta t\int_\Omega(a(Du_1^h)-a(Dv_1^h))\cdot(Du_1^h-Dv_1^h){\dd}x\leq \|u_0^h-v_0^h\|_{L^2(\Omega;\R{m})}^2.\]
As the functions $u_1^h$ and $v_1^h$ are expressed in terms of a finite basis, we see through the equivalence of norms in finite-dimensional spaces and the first inequality in Lemma \ref{discretebounds1} that $\|Du_1^h\|_{L^\infty(\Omega;\R{m \times n})}$ and $\|Du_1^h\|_{L^\infty(\Omega;\R{m \times n})}$ are both bounded by a constant $C=C(h,\|u_0\|_{L^2(\Omega;\R{m})},\|F\|_{L^2(Q_T;\R{m})})$. Therefore we can make use of the local Lipschitz condition satisfied by $a$ and the inverse inequality
\[ \|Du_1^h-Dv_1^h\|_{L^2(\Omega;\R{m})} \leq C(h) \|u_1^h-v_1^h\|_{L^2(\Omega;\R{m})},\]
(with $C(h) = Ch^{-1}$ when the family of triangulations $\mathcal{T}^h$ is quasiuniform) to deduce, by suppressing in our notation for the constant $C(h,\|a\|_{\rm Lip,loc})$ below its dependence on the (fixed) data $\|u_0\|_{L^2(\Omega;\R{m})}$ and $\|F\|_{L^2(Q_T;\R{m})}$, that
\[
\|u_1^h-v_1^h\|_{L^2(\Omega;\R{m})}^2(1-C(h,\|a\|_{\rm Lip,loc})\Delta t)\leq \|u_0^h-v_0^h\|_{L^2(\Omega;\R{m})}^2.
\]
Therefore by choosing $\Delta t$ sufficiently small for $h \in (0,h_0]$ fixed, so that
\[0<1-C(h,\|a\|_{\rm Lip,loc})\Delta t<1,\]
we have thus shown the desired continuous dependence on the initial data for a single time step. By iterating this estimate, noting that the constant $C(h,\|a\|_{\rm Lip,loc})$ remains the same regardless of the choice of the time level $i \in \{1,\dots, N\}$, we obtain that, for each $i=1,\ldots,N$, the following estimate holds: %
\begin{equation}\label{contdep1}
\|u^h_i-v^h_i\|_{L^2(\Omega;\R{m})}^2(1-C(h,\|a\|_{\rm Lip,loc})\Delta t)^i\leq\|u_0^h-v_0^h\|_{L^2(\Omega;\R{m})}^2.
\end{equation}
Recall that $\Delta t$ and $N$ are related via the identity $T=N\Delta t$. Hence,
\[(1-C(h,\|a\|_{\rm Lip,loc})\Delta t)^N=\left(1-\frac{TC(h,\|a\|_{\rm Lip,loc})}{N}\right)^N \rightarrow {\rm e}^{-TC(h, \|a\|_{\rm Lip,loc})}\qquad \mbox{as $N \rightarrow \infty$}.
\]
To transfer this continuous dependence on the initial data to the limiting functions $u^h$ and $v^h$ as $\Delta t \rightarrow 0_+$, we consider $u^h_{\Delta t}$ and $v^h_{\Delta t}$, defined, for $i=1,\dots,N$, by
\[u^h_{\Delta t}(t,x)=\frac{t-t_{i-1}}{\Delta t}u^h_i(x)+\frac{t_i-t}{\Delta t}u^h_{i-1}(x)\quad \text{for}\ t\in[t_{i-1},t_i]\ \text{and}\ x \in \overline\Omega,\]
and, analogously,
\[
v^h_{\Delta t}(t,x)=\frac{t-t_{i-1}}{\Delta t}v^h_i(x)+\frac{t_i-t}{\Delta t}v^h_{i-1}(x)\quad \text{for}\ t\in[t_{i-1},t_i] \ \text{and}\ x \in \overline\Omega.
\]
It then follows from \eqref{contdep1} that
\begin{equation}\label{contdep2}
\sup_{t\in[0,T]}\|u^h_{\Delta t}-v^h_{\Delta t}\|_{L^2(\Omega;\R{m})}^2(1-C(h,\|a\|_{\rm Lip,loc})\Delta t)^N
\leq\|u_0^h-v_0^h\|_{L^2(\Omega;\R{m})}^2.
\end{equation}

Let $\Delta t_k$ be a sequence such that $u^h_{\Delta t_k}\rightarrow u^h$ and $v^h_{\Delta t_k}\rightarrow v^h$ in $C([0,T];L^2(\Omega;\R{m}))$ and $k\rightarrow\infty$ (cf. Lemma \ref{weaktimeconvergence}). Now passing $k\rightarrow\infty$ (whereby $\Delta t_k = T/N_k \rightarrow 0_+$ and $N_k \rightarrow \infty$) we see that
\[\|u^h-v^h\|_{C([0,T];L^2(\Omega;\R{m}))}\leq {\rm e}^{\frac{1}{2}TC(h,\|a\|_{\rm Lip,loc})}\|u_0^h-v_0^h\|_{L^2(\Omega;\R{m})}.\]
We have therefore shown the following result.

\begin{proposition}\label{contmapping}
Let all of the assumptions of Theorem \ref{YMS} hold, with the additional assumption that, for $h \in (0,h_0]$ fixed, $\Delta t = T/N$ is chosen so that $0 < 1 - C(h,\|a\|_{\rm Lip,loc})\Delta t < 1$. Then, there is a unique solution to \eqref{Galeq}, which depends continuously on the choice of discretized initial data, and which can be constructed numerically by the approximations considered in this subsection.
\end{proposition}

\begin{rem}
The above proposition in particular implies that the map $S^h_t$ sending the initial condition $u_0^h$ to $S^h_tu_0^h=u^h(t,\cdot)$ is (Lipschitz) continuous from $V^h_m$ to $V^h_m$ for all $t \in [0,T]$ and $h \in (0,h_0]$ fixed, and by the equivalence of norms in finite-dimensional vector spaces, so is the map sending the initial condition $u_0^h$ to $Du^h(t,\cdot)$. Recall that we are writing the functions $u^h$ in terms of a finite element basis
\[u^h(t,x)=\sum_{i=1}^{\mathcal{N}(h)}\alpha_i^h(t)\phi_i^h(x).\]
The above considerations enable us to view the map from \R{\mathcal{N}(h)} to \R{\mathcal{N}(h)}, which sends the vector $\underline\alpha^h(0)$ of coefficients of the discretized initial condition in terms of the finite-element basis to the vector $\underline\alpha^h(t)$ of coefficients of $u^h$, in the same finite-element basis, as a continuous function. To see this, let $\underline\alpha^h$ and $\underline\beta^h$ be the coefficients of $u^h$ and $v^h$ respectively, and let
\[ M_{i,j}^h:= \frac{(\phi^h_i,\phi^h_j)}{\|\phi^h_i\|_{L^2(\Omega;\R{m})} \|\phi^h_j\|_{L^2(\Omega;\R{m})}}, \qquad i,j=1,\dots,\mathcal{N}(h).\]
The (normalized) Gramm matrix $M^h=(M_{i,j}^h)_{i,j=1}^{\mathcal{N}(h)}$ is symmetric positive definite, with
$C_0 I \leq M^h \leq  C_1 I$, where $0<C_0 \leq C_1$, in the sense of symmetric positive definite matrices. By expressing
\[u^h(t,x)-v^h(t,x)=\sum_{i=1}^{\mathcal{N}(h)}(\alpha^h_i(t)-\beta^h_i(t))\phi^h_i(x),\]
it follows that,
\begin{align*}
C_0\min_{i=1,\ldots,\mathcal{N}(h)}\|\phi^h_i\|_{L^2(\Omega;\R{m})}\max_{t\in[0,T]}
&\left(\sum_{i=1}^{\mathcal{N}(h)}|\alpha^h_i(t)-\beta^h_i(t)|^2\right)^\frac{1}{2}
\\
&\leq C_0\max_{t\in[0,T]}\left(\sum_{i=1}^{\mathcal{N}(h)}|\alpha^h_i(t)-\beta^h_i(t)|^2\int_\Omega |\phi^h_i(x)|^2{\dd}x\right)^\frac{1}{2}\\
& \leq
\max_{t\in[0,T]}\left(\sum_{i,j=1}^{\mathcal{N}(h)}(\alpha^h_i(t)-\beta^h_i(t))(\alpha^h_j(t)-\beta^h_j(t))\int_\Omega \phi^h_i(x) \cdot \phi^h_j(x){\dd}x\right)^\frac{1}{2}\\
&=\max_{t\in[0,T]}\left(\int_\Omega \sum_{i,j=1}^{\mathcal{N}(h)}(\alpha^h_i(t)-\beta^h_i(t))\phi^h_i(x)(\alpha^h_j(t)-\beta^h_j(t))\phi^h_j(x) {\dd}x\right)^\frac{1}{2}%\\
\end{align*}
\begin{align*}
&=\max_{t\in[0,T]}\left(\int_\Omega \left|\sum_{i=1}^{\mathcal{N}(h)}(\alpha^h_i(t)-\beta^h_i(t) \phi^h_i(x)\right|^2{\dd}x\right)^\frac{1}{2}\\
& = \|u^h-v^h\|_{C([0,T];L^2(\Omega;\R{m}))} \leq C(T,h,\|a\|_{\rm Lip,loc})\|u_0^h-v_0^h\|_{L^2(\Omega;\R{m})}\\
&\leq C(T,h,\|a\|_{\rm Lip,loc})\, C_1 \max_{i=1,\ldots,\mathcal{N}(h)}\|\phi^h_i\|_{L^2(\Omega;\R{m})}
\left(\sum_{i=1}^{\mathcal{N}(h)}|\alpha^h_i(0)-\beta^h_i(0)|^2\right)^\frac{1}{2}. \end{align*}
Thus the map $\underline{\alpha}^h(0) \in \R{\mathcal{N}(h)}\mapsto\underline{\alpha}^h(t)\in \R{\mathcal{N}(h)}$ is (Lipschitz) continuous for all $t\in[0,T]$ and $\Delta t$ sufficiently small so as to ensure that $0 < 1 - C(h,\|a\|_{\rm Lip,loc})\Delta t < 1$, with $h \in (0,h_0]$ fixed.
\end{rem}

\section{Computation of Young measure solutions}\label{sec:5}

{Theorem \ref{YMS} and the results of Section \ref{sec:4} (in particular, Proposition \ref{contmapping}) give criteria under which Young measure solutions exist as limits of solutions to corresponding semi- or fully-discrete problems. However, the results of Section \ref{sec:3} and Section \ref{sec:4} only give ``indirect" information about the Young measure $\nu$ through its action on functions belonging to the space $E_{\frac{p}{\alpha}}$ for $\alpha\in(1,\frac{p}{q-1})$. Indeed, note that for any function $b$ which has the form \eqref{modelstructure}--\eqref{modelconstants}, we can copy the calculations used to deduce \eqref{a-lmit-b} to see that $b(Du^h)\rightharpoonup \langle \nu,b\rangle$ in $L^{\hat{q}'}(Q_T;\R{m})$.} In this section we shall discuss an algorithm for the numerical approximation of Young measure solutions to systems of the form \eqref{modelprob1}, {which allows for a more direct approximation of the Young measure $\nu$}. The ideas presented here have been inspired by \cite{FKMT} and \cite{FMT}, where the authors develop and analyze numerical schemes for the approximation of measure-valued solutions to systems of hyperbolic conservation laws. We will show that the algorithms developed there can be adapted to systems of forward-backward parabolic equations of the form \eqref{modelprob1} exhibiting Young measure solutions. In particular, we shall demonstrate that, for the class of problems under consideration here,  there are Young measure solutions for which the measure $\nu$ may be constructed as a limit of averaged sums of particular Dirac masses, each of whose support is the value taken by an appropriate approximating solution. Throughout this section, we shall be taking subsequences to approximate Young measure solutions, for which we have not shown a uniqueness result: therefore, as we have already noted in the Introduction, different subsequences could potentially converge to different Young measure solutions.
%and we have no notion of uniqueness,
%this makes the results in this chapter difficult to use for actual computations. What follows in this chapter is therefore another method to approximate Young measure solutions.
%We shall use the term “construction” as opposed to “computation” to indicate this, and we keep in mind that we are not %computing a particular Young measure solution, nor are we
%approximating a unique Young measure solution.
%We further remark, that, as there is non-uniqueness of these Young measure solutions, the solutions constructed in this section may not be the same as those constructed in Section \ref{sec:3}.

\subsection{Preliminary definitions and results}

We recall here a selection of definitions and results from Appendix 1 of \cite{FKMT} that are pertinent to the discussion herein, extending them, where necessary.

\begin{definition}
We let $\mathcal{P}(\R{m})$ denote the set of all probability measures on \R{m}, and define, for $r>0$, the subset $\mathcal{P}^r(\R{m})$ of all probability measures $\mu \in \mathcal{P}(\R{m})$ for which $\langle \mu,|\xi|^r\rangle<\infty$, where $\xi$ denotes the identity function.
\end{definition}

We denote by $(\mathfrak{X},\mathcal{F},P)$ a probability space, with $\mathcal{F}$ being a $\sigma$-algebra of sets on the space $\mathfrak{X}$, and $P$ being a probability measure, and we let $u:\mathfrak{X}\times Q_T\rightarrow\R{m}$ be a random field (that is, a jointly measurable function). Given this, we define the law of $u$ as follows:
\begin{equation}\label{law}
\mu_{t,x}(B):=P(\{\omega\in \mathfrak{X}:\ u(w;t,x)\in B\}),
\end{equation}
for Borel subsets $B\subset\R{m}$. It is clear that $\mu_{t,x}(\R{m})=1$ for all $(t,x) \in Q_T$. We have the following result.

\begin{lemma}\label{lawequiv}
Let $(\omega;t,x) \in \mathfrak{X}\times Q_T \mapsto u(\omega;t,x) \in \R{m}$ be a random field; then, \eqref{law} is equivalent to
\[\langle\mu_{t,x},g\rangle=\int_\mathfrak{X} g(u(\omega;t,x)){\dd}P(\omega)\qquad \mbox{for a.e. $(t,x) \in Q_T$},\]
for every continuous $r$-component vector function $g$, defined on $\R{m}$, which is such that $\int_\mathfrak{X} |g(u(\omega;t,x))|{\dd}P(\omega)$ is finite for a.e. $(t,x) \in Q_T$.
\end{lemma}
\begin{proof} As the asserted equality is understood componentwise, it suffices to prove it component-by-component. We
shall therefore assume in the argument below that $r=1$, i.e., that $g$ is a continuous mapping from $\R{m}$ into $\R{}$
such that $\int_\mathfrak{X} |g(u(\omega;t,x))|{\dd}P(\omega)$ is finite for a.e. $(t,x) \in Q_T$. Let us first consider the case when the function $g:\R{m} \rightarrow \R{}$ is nonnegative, and take a sequence of simple functions $g_n$ defined on \R{m} by
\[g_n(\xi)=\sum_{k=1}^{n2^n+1} b^n_k\chi_{A^n_k}(\xi),\qquad \xi \in \R{m},\quad n=1,2,\ldots,\]
where $A^n_k$, $k=1,\dots, n$, are disjoint measurable subsets of $\R{m}$, $n$ is a positive integer, and $b^n_k \in \R{}$, $k=1,\dots,n$, are such that $g_n(\xi)$ increases to $g(\xi)$. This can be achieved by defining
\[ A^n_k:= g^{-1}\left(\left[\frac{k-1}{2^n},\frac{k}{2^n}\right)\right), \qquad 1 \leq k \leq n 2^n,\]
and
\[ A_{n,n2^n+1}:=g^{-1}([n,\infty)).\]
Observe that $g_n(\xi) \leq g(\xi)$ for all $\xi \in \R{m}$. Further, for a given $\xi \in \R{m}$ we have that $g_n(\xi) = n$ for $g(\xi)\geq n$ whereas for $g(\xi)<n$ we have that
\[ g_n(\xi)= \frac{1}{2^n}\lfloor 2^n g(\xi)\rfloor,\]
where $\lfloor y \rfloor$ denotes the greatest integer $\leq y$. As $\lfloor 2^{n+1} g(\xi)\rfloor \geq 2 \lfloor 2^n g(\xi)\rfloor$, it follows that $g_{n+1}(\xi) \geq g_n(\xi)$ for all $\xi \in \R{m}$. Furthermore, as $\lfloor 2^n g(\xi)\rfloor \geq 2^n g(\xi) - 1$, we have that $g_n(\xi) \geq g(\xi) - 2^{-n}$ as soon as $n>g(\xi)$. Since $g(\xi) \geq g_n(\xi)$ it follows that $\lim_{n \rightarrow \infty} g_n(\xi) = g(\xi)$ for all $\xi \in \R{m}$.

\smallskip

Now, let $\mathcal{F}^n_k:=\{\omega\in \mathfrak{X}:\ u(\omega;t,x)\in A^n_k\}$. We observe that
\[
\chi_{\mathcal{F}^n_k}(\omega)
=\begin{cases}
1 & \mbox{for $\omega\in \mathcal{F}^n_k$}\\
0 & \mbox{for $\omega\notin \mathcal{F}^n_k$}
\end{cases}
\,=\,
\begin{cases}
1 & \mbox{for $u(\omega;t,x)\in A^n_k$}\\
0 & \mbox{for $u(\omega;t,x)\notin A^n_k$}
\end{cases}
\,=\,
\chi_{A^n_k}(u(\omega;t,x)).
\]
Hence we have that
\begin{align*}\langle \mu_{t,x},g\rangle&= \int_{\R{m}} g(\xi){\dd}\mu_{t,x}(\xi) =\lim_{n\rightarrow\infty}\int_{\R{m}}g_n(\xi){\dd}\mu_{t,x}(\xi) =\lim_{n\rightarrow\infty}\sum_{k=1}^{n2^n+1}b^n_k\mu_{t,x}(A^n_k)\\
&=\lim_{n\rightarrow\infty}\sum_{k=1}^{n2^n+1} b^n_kP(\mathcal{F}^n_k)
=\lim_{n\rightarrow\infty}\int_\mathfrak{X}\sum_{k=1}^{n2^n+1}b^n_k\chi_{\mathcal{F}^n_k}(\omega){\dd}P(\omega)\\
&=\lim_{n\rightarrow\infty}\int_\mathfrak{X}\sum_{k=1}^{n2^n+1}b^n_k\chi_{A^n_k}(u(\omega;t,x)){\dd}P(\omega)\\
&=\lim_{n\rightarrow\infty} \int_\mathfrak{X} g_n(u(\omega;t,x)){\dd}P(\omega) =\int_\mathfrak{X} g(u(\omega;t,x)){\dd}P(\omega),
\end{align*}
where we have applied the Monotone Convergence Theorem in the first and last lines to exchange the limit with the integral over $\R{m}$. From here the extension to any, not necessarily nonnegative, function $g:\R{m} \rightarrow \R{}$ such that $\int_\mathfrak{X} |g(u(\omega;t,x))|{\dd}P(\omega)$ is finite for a.e. $(t,x) \in Q_T$, is straightforward: we decompose $g$ as $g=g^+ - g^-$, where $g^+:=\frac{1}{2}(|g| + g)$ and $g^-:=\frac{1}{2}(|g| - g)$ are (respectively) the positive and negative part of $g$, and apply the above argument to $g^+$ and $g^-$ separately, noting that both are nonnegative on $\R{m}$.
\end{proof}
\begin{rem} \label{rem:extend}
It is easy to extend Lemma \ref{lawequiv} to functions of the form $g(t,x,u(\omega,t,x))$ where $g$ is a Carath\'eodory function, i.e., it is continuous in its third variable for a.e. $(t,x) \in Q_T$ and measurable in $(t,x)\in Q_T$ for each value of the third variable, and $\int_\mathfrak{X} |g(t,x,u(\omega,t,x))|\dd P(\omega) < \infty$ for a.e. $(t,x) \in Q_T$.
\end{rem}

We have the following proposition, which shows that the law $\mu=\mu_{t,x}$ defined by \eqref{law} is in fact a Young measure. This is Proposition 1 in Appendix 1 of \cite{FKMT}.

\begin{proposition}
If $u:\mathfrak{X}\times Q_T\rightarrow\R{m}$ is jointly measurable, then the law $\mu=\mu_{t,x}$ (cf. \eqref{law}) defines a Young measure.
\end{proposition}
\begin{proof}
The proof is the same as that in \cite{FKMT}.
\end{proof}

Finally we consider the following result adapted from \cite{FKMT}, found therein as Proposition 2 in Appendix 1.

\begin{proposition}\label{randvar}
For every Young measure $\{\nu_{t,x}\}_{(t,x)\in Q_T}\in \mathcal{P}^r$, $r>0$, on \R{m} there exists a probability space $(\mathfrak{X},F,P)$ with $P\in\mathcal{P}^r$ and a Borel measurable function $u:\mathfrak{X}\times\Omega\rightarrow\R{m}$ such that $u$ has law given by $\nu$. In particular we may choose the probability space to be the Lebesgue measure on the interval $[0,1)$ with the Borel $\sigma$-algebra.
\end{proposition}
\begin{proof}
The proof of this result relies on using the characterization given in Lemma \ref{lawequiv}. In \cite{FKMT} a proof is supplied with $g\in C_0(\R{m})$; we extend this result below to cover our particular class of Young measures. We give the proof in the case of scalar-valued functions $u$; the extension to vector-valued functions follows by a component-wise argument.

To extend \emph{Proposition 2} in Appendix 1 of \cite{FKMT} from $g\in C_0(\R{})$ to $g\in C(\R{})$ with $\int_\mathfrak{X} |g(u(\omega;t,x))|{\dd}P(\omega)<\infty$ for a.e. $(t,x) \in Q_T$, let $\varphi \in C^\infty_0(\R{})$ be such that $0 \leq \varphi \leq 1$, $\varphi(\xi) \equiv 1$ for $|\xi| \leq 1$ and $\varphi(\xi) \equiv 0$ for $|\xi|\geq 2$. Let $g_k(x):= g(x) \varphi(x/k)$. Clearly, $g_k$ is a sequence of functions in $C_0(\R{})$, $g_k(\xi)=g(\xi)$ if $|\xi|\leq k$, and $g_k$ converges pointwise to $g$ on $\R{}$ as $k \rightarrow \infty$. Furthermore, $|g_k(x)| \leq |g(x)|$ for all $x \in \R{}$. Thus, thanks to the Dominated Convergence Theorem, we have that
\begin{align*}
\int_\mathfrak{X} g(u(\omega;x)){\dd}P(\omega)&=\lim_{k\rightarrow\infty}\int_\mathfrak{X} g_k(u(\omega;x)){\dd}P(\omega)
=\lim_{k\rightarrow\infty}\int_{\R{}}g_k(\xi){\dd}\nu_x(\xi)
=\int_{\R{}}g(\xi){\dd}\nu_x(\xi).
\end{align*}
We remark that progressing from the left-hand side of the second equality above to the right-hand side of that equality is precisely the result of \emph{Proposition 2} in Appendix 1 of \cite{FKMT}.
\end{proof}

Finally, we state and prove a standard lemma regarding independent and identically distributed random variables, and their images under measurable functions.

\begin{lemma}\label{iid}
Let $(\mathfrak{X},F,P)$ be a probability space and let $f:\mathfrak{X}\rightarrow\R{m}$ be a measurable function. If two random variables, $Y_1$ and $Y_2$, defined on $\mathfrak{X}$ are independent and identically distributed, then the random variables $f(Y_1)$ and $f(Y_2)$ are independent and identically distributed on $\mathfrak{X}$.
\end{lemma}
\begin{proof}
We first prove that the random variables $f(Y_1)$ and $f(Y_2)$ are identically distributed. We let $\mathcal{A}\subset \R{m}$ be a measurable set. It then follows from the measurability of $f$ that
\begin{align*}
P(f(Y_1(\omega))\in\mathcal{A})&=P(Y_1(\omega)\in f^{-1}(\mathcal{A}))
=P(Y_2(\omega) \in f^{-1}(\mathcal{A}))
=P(f(Y_2(\omega))\in \mathcal{A})\qquad \forall\, \omega \in \mathfrak{X}.
\end{align*}
To show independence, we let $\mathcal{A}_1$ and $\mathcal{A}_2$ be measurable sets and calculate
\begin{align*}
P\left((f(Y_1(\omega))\in\mathcal{A}_1)\cap(f(Y_2(\omega))\in\mathcal{A}_2)\right)&=P\left((Y_1(\omega)\in f^{-1}(\mathcal{A}_1))\cap(Y_2(\omega)\in f^{-1}(\mathcal{A}_2))\right)\\
&=P(Y_1(\omega)\in f^{-1}(\mathcal{A}_1))\,P(Y_2(\omega)\in f^{-1}(\mathcal{A}_2))\\
&=P(f(Y_1(\omega))\in\mathcal{A}_1)\,P(f(Y_2(\omega))\in\mathcal{A}_2)\qquad \forall\, \omega \in \mathfrak{X}.
\end{align*}
That completes the proof of the lemma.
\end{proof}

\subsection{An overview of the algorithm}

The algorithms discussed below involve generating (at random) a large set of initial data, evolving each of these forward under the solution operator of a semi- or fully-discrete numerical method, and considering the arithmetic average of the resulting functions at a fixed time, before completing various limit passages. The algorithms considered here have been motivated by similar algorithms used in \cite{FKMT} and \cite{FMT} for the approximation of measure-valued solutions to hyperbolic problems. To this end, let $Y(\Omega,\R{m})$ denote the set of all Young measures from $\Omega$ to $\R{m}$.

\bigskip

\noindent
\textbf{Algorithm A:} Let the initial data for an underlying time-dependent PDE be given as a Young measure $\sigma \in Y(\Omega;\R{m})$ and let $\Delta$ denote a PDE discretization parameter associated with a certain numerical scheme (in \cite{FKMT} and \cite{FMT} a finite difference scheme is utilized, and so $\Delta$ would correspond to a vector of grid sizes in the various co-ordinate directions).

\medskip

\noindent \textbf{Step 1:} Let $u_0: \omega \in \mathfrak{X} \mapsto u_0(\omega;\cdot) \in L^p(\Omega;\R{m})$ be a random field on a probability space $(\mathfrak{X},F,P)$ with law $\sigma_x$, meaning that $\sigma_x(E)=P(u_0(\omega;x)\in E)$ for all Borel sets $E \subset \R{m}$ and $x \in \Omega$.

\smallskip

\noindent \textbf{Step 2:} Evolve the initial random field by applying a suitable numerical scheme, with solution map $\textbf{S}_t^{\Delta}$, to the initial data $u_0(\omega;\cdot)$ for every $\omega\in \mathfrak{X}$, obtaining an approximate random field $u^\Delta(\omega;\cdot,t):=\textbf{S}_t^{\Delta}u_0(\omega;\cdot)$, $t \in (0,T]$.

\smallskip

\noindent \textbf{Step 3:} Define the approximate measure-valued solution $\mu^{\Delta}$ as the law of $u^\Delta$ with respect to $P$, that is, for all Borel sets $E\subset\R{m}$ and $(t,x) \in Q_T$,
\[\mu_{t,x}^\Delta(E)=P(u^\Delta (\omega;t,x)\in E).\]

\smallskip

Note that we should expect a slight difference compared with \cite{FKMT} and \cite{FMT} in what we are trying to do, as the Young measures appearing in our context are generated by sequences of gradients. Next, we need a method to approximate the random field $(\omega;\cdot,\cdot) \mapsto u^\Delta(\omega;\cdot,\cdot)$, as performing these computations for every $\omega\in \mathfrak{X}$ would be infeasible. This is \emph{Algorithm 4.3} in \cite{FMT}.

\medskip

\noindent \textbf{Algorithm B:} Let, as above, $\Delta$ denote a discretization parameter, and let $M\in\mathbb{N}$. Let $\sigma^{\Delta}$ be the initial Young measure.

\medskip

\noindent \textbf{Step 1:} From some probability space $(\mathfrak{X},F,P)$ draw $M$ independent and identically distributed random fields $u_0^{\Delta,1},\ldots,u_0^{\Delta,M}$ all with the same law $\sigma^\Delta$.

\smallskip

\noindent \textbf{Step 2:} For each $k \in \{1,\dots,M\}$ and for a fixed $\omega\in \mathfrak{X}$ approximate the solution to the PDE using the solution operator with initial data $u_0^{\Delta,k}(\omega)$; denote $u^{\Delta,k}(\omega;\cdot,t):=\textbf{S}_t^{\Delta}u_0^{\Delta,k}(\omega;\cdot)$.

\smallskip

\noindent \textbf{Step 3:} Define the approximate measure-valued solution by
\[\mu_{t,x}^{\Delta,M}:=\frac{1}{M}\sum_{k=1}^M\delta_{u^{\Delta,k}(\cdot;t,x)}.\]
\medskip

Results from \cite{FMT} (see \emph{Theorem 4.5}, \emph{Theorem 5.1} and \emph{Corollary 5.4}) then guarantee convergence of the sequence of approximate measure-valued solutions to a measure-valued solution as one passes to the limit (diagonally). In this section we discuss in more detail these algorithms and their adaptation to the systems of parabolic PDEs which we are interested in.

\subsection{Modifications for systems of forward-backward parabolic PDEs}

We now describe the necessary changes to Algorithms A and B described above in order to be able to apply them to the problem \eqref{modelprob1}--\eqref{modelconstants}.
%\textcolor{red}{As has already been discussed at the start of this chapter, repeatedly taking subsequences means that the results are hard to use for actual computation, and as such these are not really computable algorithms, but we use the word in order to be consistent with terminology used in \cite{FKMT} and \cite{FMT}.}

Below we formulate the analogues of Algorithm A and Algorithm B, which are needed for our parabolic problem. Since in our case the solution is a Sobolev function, we make here the additional restriction that we only consider atomic initial data: that is, our initial datum is assumed to be given by a function $u_0 \in L^2(\Omega;\R{m})$, which we view as the atomic Young measure
$\delta_{u_0}=\delta_{u_0(x)}$, $x \in \Omega$.

\bigskip

\noindent \textbf{Algorithm C:} Let the initial datum for the problem \eqref{modelprob1}--\eqref{modelconstants} be given as a function $u_0\in L^2(\Omega;\R{m})$ and let $h \in (0,h_0]$ be the spatial grid size parameter.

\medskip

\noindent \textbf{Step 1:} Let $\upsilon:\mathfrak{X}\rightarrow L^2(\Omega;\R{m})$ be a random field on a probability space $(\mathfrak{X},F,P)$, and discretize $\upsilon$ by a finite element approximation of random fields $\upsilon^h \in V^h_m$, so that $\|\upsilon^h(\omega;\cdot)\|_{L^2(\Omega;\R{m})}\leq1$ for $P$-a.e. $\omega \in \mathfrak{X}$. Discretize $u_0$ by a finite element approximation $u_0^h \in V^h_m$, and then perturb this discretization by defining
\[u_0^{h,\epsilon}(\omega;x):=u_0^h(x)+\epsilon \upsilon^h(\omega;x),\qquad \omega \in \mathfrak{X},\quad x \in \Omega,\]
where $\epsilon \in (0,1]$, and let $\sigma^{h,\epsilon}$ be the law of $u_0^{h,\epsilon}$, meaning that $\sigma^{h,\epsilon}_x(E)=P(u_0^{h,\epsilon}(\omega;x)\in E)$ for $\omega \in \mathfrak{X}$ and $x \in \Omega$.

\smallskip

\noindent \textbf{Step 2:} Evolve the initial random field by applying a suitable numerical scheme, with solution map $\textbf{S}_t$, to the initial datum $u_0^{h,\epsilon}(\omega)$ for every $\omega\in \mathfrak{X}$, obtaining an approximate random field $u^{h,\epsilon}(\omega,t,\cdot):=\textbf{S}_tu_0^{h,\epsilon}(\omega;\cdot)$, $t \in (0,T]$.

\smallskip

\noindent \textbf{Step 3:} Define the approximate Young measure solution $\mu^{h,\epsilon}$ as the law of $u^{h,\epsilon}$ with respect to $P$, that is, for all Borel sets $E\subset\R{m}$, %
\[\mu_{t,x}^{h,\epsilon}(E)=P(u^{h,\epsilon}(\omega;t,x)\in E),\qquad \omega \in \mathfrak{X},\quad (t,x) \in Q_T, \]
and define, analogously, $\nu^{h,\epsilon}$ to be the law of $Du^{h,\epsilon}$ with respect to $P$.

\medskip

As before we need a method to approximate the random field $u_0^\epsilon(\omega;x)$, as well as the measures $\mu^{h,\epsilon}$ and $\nu^{h,\epsilon}$.

\medskip

\noindent \textbf{Algorithm D:} Let the initial datum for the problem \eqref{modelprob1}--\eqref{modelconstants} be given as a function $u_0\in L^2(\Omega;\R{m})$, let $h \in (0,h_0]$ be the spatial grid size parameter, and let $M\in\mathbb{N}$.

\medskip

\noindent \textbf{Step 1:} From some probability space $(\mathfrak{X}, F, P)$ draw $M$ independent and identically distributed random fields $\upsilon^{h,1},\ldots,\upsilon^{h,M}$ such that $\|\upsilon^{h,i}(\omega;\cdot)\|_{L^2(\Omega;\R{m})}\leq1$, for $P$-a.e. $\omega\in \mathfrak{X}$ and for all $i=1,\ldots,M$, and such that $u_0^{h,1,\epsilon},\ldots,u_0^{h,M,\epsilon}$ all have the same law $\sigma^{h,\epsilon}$, where
\[u_0^{h,k,\epsilon}(\omega;x)=u_0^h(x)+\epsilon \upsilon^{h,k}(\omega;x),\qquad \omega \in \mathfrak{X}, \quad x \in \Omega,\]
with $\epsilon \in (0,1]$. We make explicit here the fact that $\|u_0^{h,k,\epsilon}(\omega;\cdot)\|_{L^2(\Omega;\R{m})}\leq \|u_0^h\|_{L^2(\Omega;\R{m})}+\epsilon$ for all $\omega \in \mathfrak{X}$, so that $\|u_0^{h,k,\epsilon}(\omega;\cdot)\|_{L^2(\Omega;\R{m})}$ can be assumed to be bounded independent of $k$, $\epsilon$, and $h$ (thanks to the assumed strong convergence of $u_0^h$ to $u_0$ in $L^2(\Omega;\R{m})$ and because $\epsilon \in (0,1]$).

\smallskip

\noindent \textbf{Step 2:} For each $k$ and for a fixed $\omega\in \mathfrak{X}$ approximate the solution to the initial boundary-value problem \eqref{modelprob1}--\eqref{modelconstants} under consideration using the solution operator with initial data $u_0^{h,k,\epsilon}(\omega)$; denote $u^{h,k,\epsilon}(\omega;t,\cdot)=\textbf{S}_tu_0^{h,k,\epsilon}(\omega;\cdot,\cdot)$.

\smallskip

\noindent \textbf{Step 3:} Define the approximate Young measure solution $(\mu_{t,x}^{h,M,\epsilon},\nu_{t,x}^{h,M,\epsilon})$
by
\[\mu_{t,x}^{h,M,\epsilon}:=\frac{1}{M}\sum_{k=1}^M\delta_{u^{h,k,\epsilon}(\cdot;t,x)} \qquad \mbox{and}\qquad \nu_{t,x}^{h,M,\epsilon}:=\frac{1}{M}\sum_{k=1}^M\delta_{Du^{h,k,\epsilon}(\cdot;t,x)}.\]

\subsection{Applying the algorithm to construct Young measure solutions}

In what follows we will show that Algorithm D described above converges to a Young measure solution of the system of PDEs \eqref{modelprob1}--\eqref{modelconstants} under consideration as we pass $\Delta t \rightarrow 0_+$, $M\rightarrow\infty$, $h,\epsilon\rightarrow0_+$ (in that order).

\begin{rem}
The parameter $\omega$ taken from the probability space $\mathfrak{X}$ is intended to represent the ``seed" in the numerical algorithm. Different seeds $\omega \in \mathfrak{X}$ will give rise to different approximate solutions. What will be shown in our analysis of the limit process $\Delta t \rightarrow 0_+$,
$M \rightarrow \infty$, $h, \epsilon \rightarrow 0_+$ is that one can pass to the limit in $M$ (along a subsequence) to obtain convergence to a quantity which is independent of $\omega\in \mathfrak{X}$, for $P$-a.e. $\omega\in \mathfrak{X}$.
\end{rem}

\smallskip

\noindent \textbf{{Step 1: Passage to the limits $\Delta t \rightarrow 0_+$, $M\rightarrow\infty$}}. We shall use the semidiscrete numerical scheme discussed in Section \ref{sec:3}, augmented with the time-stepping procedure discussed in Section \ref{sec:4}. By Lemma \ref{contmapping} we have that, for any $\omega\in \mathfrak{X}$, any $\epsilon>0$, any positive integer $k$ and any $h \in (0,h_0]$ and $\Delta t = T/N$, $N \in \mathbb{N}$, we can construct, via the time-stepping procedure from Section \ref{sec:4}, a function $u^{h,k,\epsilon}_{\Delta t}(\omega;\cdot,\cdot)$, which for any $h \in (0,h_0]$ and any $\Delta t>0$ (sufficiently small, for a given fixed $h \in (0,h_0]$) uniquely solves the fully discrete scheme and depends continuously on the discretized initial condition $u_0^{h,k,\epsilon}$. For the rest of the analysis we shall assume for the sake of brevity that the passage to the limit $\Delta t\rightarrow0_+$ has already been made for $\omega \in \mathfrak{X}$ and $h \in (0,h_0]$ fixed,
by repeating the analysis performed in Section \ref{sec:4}, and we shall therefore consider the limiting function $u^{h,k,\epsilon}(\omega;\cdot,\cdot)$ resulting from the passage to the limit $\Delta t \rightarrow 0_+$ instead of $u^{h,k,\epsilon}_{\Delta t}(\omega;\cdot,\cdot)$. Passage to the limit $\Delta t \rightarrow 0_+$, for reasons that will be explained in due course, must happen before the limit passage $M\rightarrow \infty$. In particular we remark that as we are at this stage considering finitely many perturbations of the discretized initial condition, the existence of a suitable subsequence $\Delta t_j\rightarrow0_+$, which works for each $k\in\{1,\ldots,M\}$, is straightforward.

For each initial datum $u_0^{h,k,\epsilon}(\omega;\cdot)$ (for fixed $\omega\in \mathfrak{X}$ and $k\in\{1,\ldots,M\}$), we have from
\eqref{modelenergy-b} the following energy estimate:
\begin{align}\label{e:kbound}
\begin{aligned}
\|u^{h,k,\epsilon}(\omega)\|_{L^\infty(0,T;L^2(\Omega;\R{m}))}^2&+\sum_{i=1}^m\
\|Du_i^{h,k,\epsilon}(\omega)\|_{L^{p_i}(0,T;W_0^{1,p_i}(\Omega;\R{m\times n}))}^{p_i}
\\
& \leq c(1+ \|u_0^{h,k,\epsilon}(\omega)\|_{L^2(\Omega;\R{m})}^2+\|F\|_{L^{p'}\!(Q_T;\R{m})}^{p'}).
\end{aligned}
\end{align}
Furthermore, \eqref{e:au-bound} implies, with $\hat{q}:=\max\{q,2\}$  and $\hat{q}'=\frac{\hat{q}}{\hat{q}-1}$, that
\begin{align}\label{e:kbound-a}
\|a(Du^{h,k,\epsilon})(\omega)\|_{L^{\hat{q}'}\!(Q_T;\R{m\times n})}\leq C
\end{align}
and \eqref{modeltimederivative} yields
\begin{align}\label{e:kbound-t}
\|\partial_tu^{h,k,\epsilon}(\omega)\|_{L^{\hat{q}'}\!(0,T;W^{-1,\hat{q}'}\!(\Omega;\R{m}))}\leq C,
\end{align}
where $C$ is a positive constant that is independent of $h,k$ and $\epsilon$. This follows from the form of the perturbed initial condition in Algorithm D, which guarantees that the perturbed initial data are bounded uniformly in $h,k$ and $\epsilon$. Summing the equations satisfied by $u^{h,k,\epsilon}(\omega;\cdot,\cdot)$ over $k=1,\ldots,M$ and dividing by $M$ we see that the functions $u^{h,k,\epsilon}$ satisfy, for all $\psi^h\in L^q(0,T;V^h_m)$, the equality:
\begin{equation}
\begin{split}\label{newweak}
&\int_0^T\int_\Omega\frac{1}{M}\sum_{k=1}^M \partial_tu^{h,k,\epsilon}(\omega;t,x)\cdot\psi^h(t,x)+
\frac{1}{M}\sum_{k=1}^M a(Du^{h,k,\epsilon}(\omega;t,x)): D\psi^h(t,x){\dd}x{\dd}t\\
&\qquad +\int_0^T\int_\Omega\frac{1}{M}\sum_{k=1}^MBu^{h,k,\epsilon}(\omega;t,x)\cdot\psi^h(t,x){\dd}x{\dd}t =\int_0^T\int_\Omega F\cdot\psi^h{\dd}x{\dd}t.
\end{split}
\end{equation}

Integration by parts in time then yields, for all $\psi^h\in W^{1,1}_0(0,T;V^h_m)$,
\begin{equation}
\label{newweak2}
\begin{split}
&\int_0^T\int_\Omega\frac{1}{M}\sum_{k=1}^M u^{h,k,\epsilon}(\omega;t,x)\cdot (\partial_t\psi^h(t,x)-B^{\rm T}\psi^h(t,x)){\dd}x{\dd}t\\
&\qquad -\int_0^T\int_\Omega \frac{1}{M}\sum_{k=1}^M a(Du^{h,k,\epsilon}(\omega;t,x)): D\psi^h(t,x)-F(t,x)\cdot\psi^h(t,x){\dd}x{\dd}t=0.
\end{split}
\end{equation}

We now wish to let $M\rightarrow\infty$ in \eqref{newweak2}. To this end, we need to prove certain convergence results that will enable passage to this limit.  We begin by considering the sequence of functions $\{\frac{1}{M}\sum_{k=1}^M a(Du^{h,k,\epsilon}(\omega;\cdot,\cdot))\}_{M \geq 1}$, for a fixed $\omega\in \mathfrak{X}$, which is,  thanks to \eqref{e:au-bound}, bounded in the $L^{\hat{q}'}\!(Q_T;\R{m \times n})$ norm, uniformly in $h \in (0,h_0]$ and $\epsilon \in (0,1]$, and as such there is a subsequence $\{M_n\}_{n \geq 1}$ and a function $\chi^{h,\epsilon}(\omega;\cdot,\cdot)\in L^{\hat{q}'}\!(Q_T;\R{m \times n})$ such that $\frac{1}{M_n}\sum_{k=1}^{M_n} a(Du^{h,k,\epsilon}(\omega;\cdot,\cdot)) \rightharpoonup\chi^{h,\epsilon}(\omega;\cdot,\cdot)$ in $L^{\hat{q}'}\!(Q_T;\R{m \times n})$, as $n \rightarrow \infty$. We aim to identify this function $\chi^{h,\epsilon}(\omega;\cdot,\cdot)$ for $\omega \in \mathfrak{X}$, by identifying the limit in some weaker space and appealing to the uniqueness of weak limits.

Let us consider the initial datum $u^{h,\epsilon}_0(\omega;x)=u^h_0(x) + \epsilon \upsilon^h(\omega;x)$ that is independent
of the initial data $u_0^{h,1,\epsilon},\ldots,u_0^{h,M,\epsilon}$, and has the same distribution as these initial data (meaning in particular that the law of $u_0^{h,\epsilon}$ is given by $\sigma^{h,\epsilon}$), and satisfies $\|\upsilon^h(\omega;\cdot)\|_{L^2(\Omega;\R{m})}\leq 1$ for $P$-almost every $\omega\in \mathfrak{X}$. We evolve this according to the algorithm for each $\omega\in \mathfrak{X}$ to obtain a function $u^{h,\epsilon}(\omega;t,x)$. We denote the law of $u^{h,\epsilon}$ by $\mu^{h,\epsilon}$, and similarly denote the law of $Du^{h,\epsilon}$ by $\nu^{h,\epsilon}$. From Theorem \ref{YMS} we see that the following regularity results hold:
\begin{align}\label{approxreg1}
u^{h,\epsilon}&\in L^\infty\left(\mathfrak{X};\bigtimes_{i=1}^mL^{p_i}(0,T;W_0^{1,p_i}(\Omega;\R{m}))\right)\cap L^\infty(\mathfrak{X}\times(0,T);L^2(\Omega;\R{m})),\\
\label{approxreg2}\partial_t u^{h,\epsilon}&\in L^\infty(\mathfrak{X};L^{\hat{q}'}\!(0,T;W^{-1,\hat{q}'}\!(\Omega;\R{m}))),\\
\label{approxreg3}a(Du^{h,\epsilon})&\in L^\infty(\mathfrak{X};L^{\hat{q}'}\!(Q_T;\R{m\times n})).
\end{align}
Analogously to \eqref{e:kbound}--\eqref{e:kbound-t}, the norms of these functions in the respective spaces can be bounded by a positive constant $C$, independent of $h$ and $\epsilon$. Next, we formulate an analogue of \emph{Theorem 11} from \cite{FKMT}.

\begin{lemma}\label{subsequenceconvergence}
Let $\alpha\in(1,\frac{p}{q-1})$ and suppose that $\gamma\in L^{\alpha'}\!(Q_T;E_{\frac{p}{\alpha}})$ (understood as an $m\times n$ matrix-valued function). Along the subsequence $\{M_n\}_{n \geq 1}$ we have
\[\int_0^T\int_\Omega\langle\nu_{\omega;t,x}^{h,M_n,\epsilon},\gamma(t,x,\cdot)\rangle{\dd}x{\dd}t
\rightarrow\int_0^T\int_\Omega\langle\nu_{t,x}^{h,\epsilon},\gamma(t,x,\cdot)\rangle {\rm d}x{\dd}t,\]
in $L^2(\mathfrak{X})$ as $n\rightarrow\infty$. In particular, for each $\gamma \in L^{\alpha'}\!(Q_T;E_{\frac{p}{\alpha}})$ there is a subsequence of $M_n$ such that for $P$-a.e. $\omega\in \mathfrak{X}$ the convergence is pointwise.
\end{lemma}
\begin{proof}
Given a random field $\eta:\mathfrak{X}\rightarrow L^1(Q_T)$ we define its expectation with respect to the probability measure $P$ as %
\[\mathbb{E}(\eta):=\int_\mathfrak{X}\eta(\omega){\dd}P(\omega).\]
Similarly to \cite{FKMT} we define, for
$k=1,\ldots,M$, $i=1,\dots, m$ and $j=1,\dots,n$ the following quantities:
\[G_{i,j}(\omega)=\int_0^T\int_\Omega \gamma_{i,j}(t,x,Du^{h,\epsilon}(\omega;t,x)){\dd}x{\dd}t,\]
and
\[G^k_{i,j}(\omega)=\int_0^T\int_\Omega \gamma_{i,j}(t,x,Du^{h,k,\epsilon}(\omega;t,x)){\dd}x{\dd}t.\]
We need to show that $G^1_{i,j},\ldots,G^{M_n}_{i,j}$ are independent and identically distributed.

We begin by remarking that as the initial data $u_0^{h,1,\epsilon},\ldots,u_0^{h,M_n,\epsilon}$ are independent and identically distributed (in the variable $\omega\in \mathfrak{X}$), and as the solution operator $\textbf{S}_t$ mapping an initial condition $\tilde{u}_0$ to $\tilde{u}(t,\cdot)$ is a continuous map (see Lemma \ref{contmapping}), the functions $u^{h,k,\epsilon}$ are all independent and identically distributed as well (see Lemma \ref{iid}). Furthermore, by the remarks following the statement of Theorem \ref{contmapping}, the functions $Du^{h,k,\epsilon}$ are also independent and identically distributed.

Recall that the functions $\gamma_{i,j}$ are measurable, they are continuous in their third argument, and satisfy the growth rate coming from the space $E_{\frac{p}{\alpha}}$. This implies that the functions $\gamma_{i,j}(t,x,Du^{h,k,\epsilon}(\omega;t,x))$ are independent and identically distributed with respect to $\omega\in \mathfrak{X}$. From Tonelli's Theorem it then follows that the functions $G^k_{i,j}$ are measurable in $\omega$, and thus by Lemma \ref{iid} are also independent and identically distributed in $\omega\in \mathfrak{X}$.

We now compute the $L^2$ error with respect to the probability measure $P$ of the following quantity: %
\[\left(\mathbb{E}(G_{i,j}(\omega))-\frac{1}{M_n}\sum_{k=1}^{M_n} G^k_{i,j}(\omega)\right)^2,\]
and as in \cite{FKMT}, using the fact that the random variables $G_{i,j}^k$ are independent and identically distributed, we can reduce this to
\begin{equation}\label{expectation}
\mathbb{E}\left(\left(\mathbb{E}(G_{i,j}(\omega))-\frac{1}{M_n}\sum_{k=1}^{M_n} G^k_{i,j}(\omega)\right)^2\right)=\frac{1}{M_n}(\mathbb{E}((G_{i,j})^2)-(\mathbb{E}(G_{i,j}))^2) \leq \frac{1}{M_n}\mathbb{E}((G_{i,j})^2).
\end{equation}
Given $\alpha\in(1,\frac{p}{q-1})$ we have that
\begin{align*}
\mathbb{E}((G_{i,j})^2)&=\int_\mathfrak{X}\left(\int_0^T\int_\Omega \gamma_{i,j}(t,x,Du^{h,\epsilon}(\omega;t,x)){\dd}x{\dd}t\right)^2{\dd}P(\omega)\\
&\leq\int_\mathfrak{X}\|\gamma_{i,j}\|_{L^{\alpha'}\!(Q_T;E_{\frac{p}{\alpha}})}^{2}
\|1+|Du^{h,\epsilon}(\omega)|^{\frac{p}{\alpha}}\|_{L^\alpha(Q_T)}^2{\dd}P(\omega).
\end{align*}

Next, we can bound the norm $\|1+|Du^{h,\epsilon}(\omega)|^{\frac{p}{\alpha}}\|_{L^\alpha(Q_T)}$ independently of $h$, $\epsilon$ and $\omega$ by using the energy estimate \eqref{modelenergy-b} (analogously to \eqref{e:kbound-a}). Thus, since $P$ is a probability measure, we have that $\mathbb{E}((G_{i,j})^2)$ is bounded, and so the right-hand side of \eqref{expectation} converges to zero as $M_n\rightarrow\infty$.

We wish to apply Lemma \ref{lawequiv} (note also Remark \ref{rem:extend}). In order to do this we must show that
\[\int_\mathfrak{X} |\gamma_{i,j}(t,x,Du^{h,\epsilon}(\omega;t,x))|{\dd}P(\omega)<\infty\qquad \mbox{for a.e. $(t,x) \in Q_T$}.\]
Note that if we show that
\[\int_\mathfrak{X}\int_0^T\int_\Omega |\gamma_{i,j}(t,x,Du^{h,\epsilon}(\omega;t,x))|{\dd}x{\dd}t{\dd}P(\omega)<\infty,\]
then the desired bound is true for almost every $(t,x)\in Q_T$, which is sufficient for our purposes. We see, using that $\gamma\in L^{\alpha'}\!(Q_T;E_{\frac{p}{\alpha}})$ with $\alpha\in(1,\frac{p}{q-1})$, that

\begin{align*}
\int_\mathfrak{X}\int_0^T\int_\Omega& |\gamma_{i,j}(t,x,Du^{h,\epsilon}(\omega;t,x))|{\dd}x{\dd}t{\dd}P(\omega)\\
&= \int_\mathfrak{X}\int_0^T\int_\Omega |\gamma_{i,j}(t,x,Du^{h,\epsilon}(\omega;t,x))|\frac{1+|Du^{h,\epsilon}(\omega;t,x)|^{\frac{p}{\alpha}}}
{1+|Du^{h,\epsilon}(\omega;t,x)|^{\frac{p}{\alpha}}}{\dd}x{\dd}t{\dd}P(\omega)\\
&\leq \int_\mathfrak{X}\int_0^T\int_\Omega \|\gamma_{i,j}(t,x,\cdot)\|_{E_{\frac{p}{\alpha}}}(1+|Du^{h,\epsilon}(\omega;t,x)|^{\frac{p}{\alpha}}){\dd}x{\dd}t{\dd}P(\omega)\\
&\leq \int_\mathfrak{X} \|\gamma_{i,j}\|_{L^{\alpha'}\!(Q_T;E_{\frac{p}{\alpha}})}\|1+|Du^{h,\epsilon}(\omega;\cdot,\cdot)|^{\frac{p}{\alpha}}
\|_{L^\alpha(Q_T)} {\dd}P(\omega)\\
&\leq\|\gamma_{i,j}\|_{L^{\alpha'}\!(Q_T;E_{\frac{p}{\alpha}})}\int_\mathfrak{X} |Q_T|^\frac{1}{\alpha}+\|Du^{h,\epsilon}(\omega;\cdot,\cdot)\|_{L^p(Q_T;\R{m \times n})}^{\frac{p}{\alpha}}{\dd}P(\omega)\\
&\leq\|\gamma_{i,j}\|_{L^{\alpha'}\!(Q_T;E_{\frac{p}{\alpha}})}\left(|Q_T|^\frac{1}{\alpha}+\|Du^{h,\epsilon}\|_{L^\infty(\mathfrak{X};L^p(Q_T;\R{m \times n}))}^{\frac{p}{\alpha}}\right),
\end{align*}
which is finite by the assumptions on $\gamma$ and by \eqref{approxreg1}. Then, by applying Lemma \ref{lawequiv}, we see that: %
\begin{align*}
\mathbb{E}(G_{i,j}(\omega))&=\int_\mathfrak{X}\int_0^T\int_\Omega \gamma_{i,j}(t,x,Du^{h,\epsilon}(\omega;t,x)){\dd}x{\dd}t{\dd}P(\omega)\\
&=\int_0^T\int_\Omega\left(\int_\mathfrak{X}\gamma_{i,j}(t,x,Du^{h,\epsilon}(\omega;t,x)){\dd}P(\omega)\right){\dd}x{\dd}t\\
&=\int_0^T\int_\Omega \langle \nu^{h,\epsilon}_{i,j;t,x},\gamma_{i,j}(t,x,\cdot)\rangle{\dd}x{\dd}t.
\end{align*}

Repeating these calculations for all components $\gamma_{i,j}$ gives us the convergence of %
\[\int_0^T\int_\Omega\langle\nu_{\omega;t,x}^{h,M_n,\epsilon},\gamma(t,x,\cdot)\rangle {\rm d}x{\dd}t\rightarrow\int_0^T\int_\Omega\langle\nu_{t,x}^{h,\epsilon},\gamma(t,x,\cdot)\rangle{\dd}x{\dd}t,\]
in $L^2(\mathfrak{X})$ as $M_n\rightarrow\infty$, from which we deduce the existence of a subsequence (not relabelled) converging for $P-$ almost every $\omega\in \mathfrak{X}$ using standard results in measure theory.
\end{proof}

What we would like to conclude is that by the above theorem we have, for $P-$almost every $\omega\in \mathfrak{X}$, the weak-star convergence of $\nu^{h,M_n,\epsilon}_\omega=\nu^{h,M_n,\epsilon}_{\omega,\cdot,\cdot}$ to $\nu^{h,\epsilon}=\nu^{h,\epsilon}_{\cdot,\cdot}$ as $n \rightarrow \infty$. This does not follow immediately, as a-priori there is \textit{no reason} why the subsequence of $M_n$ along which we have pointwise convergence should be \textit{independent} of the function $\gamma$ (a point not clarified in \cite{FKMT} or \cite{FMT}). To get around this we note that as $E_{\frac{p}{\alpha}}$ is separable (see \cite{KinPed2}), the space $L^{\alpha'}\!(Q_T;E_{\frac{p}{\alpha}})$ is also separable for our choice of $\alpha$. Therefore there is a countable dense subset $\{\gamma_i\}_{i=1}^\infty$. The desired weak-star convergence then follows from a diagonal argument.

Recall from the calculations performed prior to Lemma \ref{subsequenceconvergence} that, as $M_n \rightarrow \infty$ ($n \rightarrow \infty$),
\begin{align}\label{weak-r1}
\int_0^T\int_\Omega \frac{1}{M_n}\sum_{k=1}^{M_n}a(Du^{h,k,\epsilon}(\omega;t,x)) : D\psi(t,x){\dd}x{\dd}t\rightarrow\int_0^T\int_\Omega \chi^{h,\epsilon}(\omega;t,x) : D\psi(t,x){\dd}x{\dd}t
\end{align}
for every $\psi\in L^{\hat{q}}([0,T];W^{1,\hat{q}}_0(\Omega;\R{m}))$, and for all $\omega \in \mathfrak{X}$. Furthermore, by Lemma \ref{subsequenceconvergence}, there is a subsequence of $M_n$ such that, as $M_n \rightarrow \infty$ ($n \rightarrow \infty$),
\begin{align}\label{weak-r2}
\int_0^T\int_\Omega \frac{1}{M_n}\sum_{k=1}^{M_n}a(Du^{h,k,\epsilon}(\omega;t,x)) : D\psi(t,x){\dd}x{\dd}t\rightarrow\int_0^T\int_\Omega \langle \nu^{h,\epsilon}_{t,x},a\rangle : D\psi(t,x){\dd}x{\dd}t
\end{align}
for every $\psi\in L^{\alpha'}\!(0,T;W^{1,\alpha'}_0\!(\Omega;\R{m}))$, and for $P$-a.e. $\omega \in \mathfrak{X}$.

As both \eqref{weak-r1} and \eqref{weak-r2} hold for $\psi\in L^{\max\{\hat{q},\alpha'\}}(0,T;W^{1,\max\{\hat{q},\alpha'\}}_0(\Omega;\R{m}))$, by the uniqueness of the weak limit we can identify $\chi^{h,\epsilon}(\omega;\cdot,\cdot) =\langle\nu^{h,\epsilon},a\rangle$ for $P$-a.e. $\omega \in \mathfrak{X}$. As the right-hand side of this equality is independent of $\omega \in \mathfrak{X}$, it follows that $\chi^{h,\epsilon}$ is also independent of $\omega$, and we shall therefore suppress the dependence on $\omega$ in our notation and write $\chi^{h,\epsilon}(t,x)$ instead of $\chi^{h,\epsilon}(\omega;t,x)$, noting that $\chi^{h,\epsilon} = \langle\nu^{h,\epsilon},a\rangle \in L^{\hat{q}'}\!(Q_T;\R{m \times n})$. Thus, by combining \eqref{weak-r1} and \eqref{weak-r2}, which guarantee weak convergence in $L^{\hat{q}'}\!(0,T;W^{1,\hat{q}'}_0\!(\Omega;\R{m}))$ and $L^{\alpha}(0,T;W^{1,\alpha}_0(\Omega;\R{m}))$, respectively, for $P$-a.e. $\omega \in \mathfrak{X}$, and therefore, by uniqueness of the weak limit, also in the function space $L^{\max\{\hat{q}',\alpha\}}(0,T;W^{1,\max\{\hat{q}',\alpha\}}_0(\Omega;\R{m}))$, for $P$-a.e. $\omega \in \mathfrak{X}$, we have that
\begin{align}\label{weak-r3}
\int_0^T\int_\Omega \frac{1}{M_n}\sum_{k=1}^{M_n}a(Du^{h,k,\epsilon}(\omega;t,x)) : D\psi(t,x){\dd}x{\dd}t\rightarrow\int_0^T\int_\Omega \langle \nu^{h,\epsilon}_{t,x},a\rangle : D\psi(t,x){\dd}x{\dd}t
\end{align}
for all $\psi \in L^{\min\{\hat{q},\alpha'\}}(0,T;W^{1,\min\{\hat{q},\alpha'\}}_0(\Omega;\R{m}))$, for $P$-a.e. $\omega \in \mathfrak{X}$.

The terms involving the measures $\mu^{h,\epsilon}$ are treated completely analogously. Thus, by passing to the limit $M_n \rightarrow \infty$ ($n \rightarrow \infty$) in \eqref{newweak2}, we have that for all $\psi^h\in W^{1,1}_0(0,T;V^h_m)$ the following identity is satisfied:
\[\int_0^T\int_\Omega \langle\mu^{h,\epsilon}_{t,x},\xi\rangle\cdot(\partial_t\psi^h(t,x)-B^{\rm T}\psi^h(t,x))-\langle\nu^{h,\epsilon}_{t,x},a\rangle : D\psi^h(t,x)+F(t,x)\cdot\psi^h(t,x){\dd}x{\dd}t=0.
\]

\medskip

\noindent \textbf{{Step 2: Passing $h,\epsilon\rightarrow0_+$}}.
By Lemma \ref{subsequenceconvergence} and noting that $a \in E_{\frac{p}{\alpha}}$, it follows that $\langle \nu^{h,M_n,\epsilon}_{\omega,\cdot,\cdot},a \rangle \rightharpoonup \langle \nu^{h,\epsilon}_{\cdot,\cdot},a \rangle$
weakly in $L^\alpha(Q_T;\R{m\times n})$, for $\alpha \in (1,\frac{p}{q-1})$, and for $P$-a.e. $\omega \in \mathfrak{X}$. Thus, by weak lower-semicontinuity of the norm function, we have that $\|\langle \nu^{h,\epsilon}_{\cdot,\cdot},a\rangle\|_{L^{\alpha}(Q_T;\R{m\times n})}\leq \liminf_{n\rightarrow\infty}\|\langle \nu^{h,M_n,\epsilon}_{\omega,\cdot,\cdot},a\rangle\|_{L^{\alpha}(Q_T;\R{m\times n})}$ for $P$-a.e. $\omega \in \mathfrak{X}$.

Let $1<\alpha<\frac{p}{q-1}$ and consider the function $\tilde{a}(\xi)=1+|\xi|^{\frac{p}{\alpha}}$; then, $\tilde{a} \in E_{p/\alpha}$. It then follows from Lemma \ref{subsequenceconvergence} that $\langle \nu^{h,M_l,\epsilon}_{\omega,\cdot,\cdot},\tilde{a}\rangle\rightharpoonup \langle \nu^{h,\epsilon},\tilde{a}\rangle$ in $L^{\alpha}(Q_T;\R{m\times n})$ as $l\rightarrow\infty$ for $P$-a.e. $\omega \in \mathfrak{X}$. Then, by weak lower-semicontinuity of the norm function $\|\cdot\|_{L^{\alpha}(Q_T;\R{m\times n})}$, we have that
\begin{align}\label{e:nu-a-bound}
\begin{aligned}
\|\langle \nu^{h,\epsilon},\tilde{a}\rangle\|_{L^{\alpha}(Q_T;\R{m \times n})}&\leq \liminf_{l\rightarrow\infty}\|\langle \nu^{h,M_l,\epsilon}_{\omega,\cdot,\cdot},\tilde{a}\rangle\|_{L^{\alpha}(Q_T;\R{m \times n})}\\
& = \liminf_{l\rightarrow\infty}\left\|\frac{1}{M_l}\sum_{k=1}^{M_l} (1 + |Du^{h,k,\epsilon}(\omega;\cdot,\cdot)|^{\frac{p}{\alpha}}) \right\|_{L^{\alpha}(Q_T)}\\
&\leq \liminf_{l\rightarrow\infty}\frac{1}{M_l}\sum_{k=1}^{M_l} \left\|1 + |Du^{h,k,\epsilon}(\omega;\cdot,\cdot)|^{\frac{p}{\alpha}}\right\|_{L^{\alpha}(Q_T)}\\
& \leq \liminf_{l\rightarrow\infty} |Q_T|^{\frac{1}{\alpha}} + \frac{1}{M_l}\sum_{k=1}^{M_l}
\left\|Du^{h,k,\epsilon}(\omega;\cdot,\cdot)\right\|_{L^{p}(Q_T;\R{m\times n})}^{\frac{p}{\alpha}}\\
& \leq \liminf_{l\rightarrow\infty} |Q_T|^{\frac{1}{\alpha}} + \frac{1}{M_l}\sum_{k=1}^{M_l}\left[c\left(1+ \|u_0^{h,k,\epsilon}(\omega)\|_{L^2(\Omega;\R{m})}^2+\|F\|_{L^{p'}\!(Q_T;\R{m})}^{p'} \right)\right]^{\frac{1}{\alpha}}\\
& \leq \liminf_{l\rightarrow\infty} |Q_T|^{\frac{1}{\alpha}} + \frac{1}{M_l}\sum_{k=1}^{M_l}\left[c\left(1+ 2\|u_0^{h}\|_{L^2(\Omega;\R{m})}^2+ 2\epsilon^2 + \|F\|_{L^{p'}\!(Q_T;\R{m})}^{p'} \right)\right]^{\frac{1}{\alpha}}\\
& =  |Q_T|^{\frac{1}{\alpha}} + c \left(3+ 2\|u_0^{h}\|_{L^2(\Omega;\R{m})}^2 + \|F\|_{L^{p'}\!(Q_T;\R{m})}^{p'} \right)^{\frac{1}{\alpha}}
\leq C,\qquad \mbox{for $P$-a.e. $\omega \in \mathfrak{X}$,}
\end{aligned}
\end{align}
where $C$ is a positive constant, independent of $h \in (0,h_0]$ and $\epsilon \in (0,1]$. Here, in the transition from the left-hand side of the second inequality to its right-hand side we have used the triangle inequality; in the transition from the left-hand side of the third inequality to its right-hand side we have again used the triangle inequality;  in the transition from the left-hand side of the fourth inequality to its right-hand side we have used the energy estimate \eqref{e:kbound}, followed by recalling that $\|u_0^{h,k,\epsilon}(\omega)\|_{L^2(\Omega;\R{m})} \leq \|u_0^{h}\|_{L^2(\Omega;\R{m})} + \epsilon$, with $\epsilon \in (0,1]$, and noting that by the assumed strong convergence of $u_0^h$ to $u_0$ in $L^2(\Omega;\R{m})$ one can bound $\|u_0^{h}\|_{L^2(\Omega;\R{m})}$ by a constant, independent of $h \in (0,h_0]$.

Having shown that $\|\langle \nu^{h,\epsilon},\tilde{a}\rangle\|_{L^{\alpha}(Q_T)}$ is bounded by a constant,
independent of $h$ and $\epsilon$, we now compute
\begin{align*}
\|\nu^{h,\epsilon}\|_{L^{\alpha}(Q_T;E_{p/\alpha}')}&=\sup_{\|\gamma\|_{L^{\alpha'}\!(Q_T;E_{p/\alpha})}=1}\ \int_{Q_T}\langle \nu^{h,\epsilon}_{t,x},\gamma(t,x,\cdot)\rangle{\dd}x{\dd}t\\
&=\sup_{\|\gamma\|_{L^{\alpha'}\!(Q_T;E_{p/\alpha})}=1}\ \int_{Q_T}\left\langle \nu^{h,\epsilon}_{t,x},\frac{\gamma(t,x,\cdot)(1+|\cdot|^{\frac{p}{\alpha}})}{1+|\cdot|^{\frac{p}{\alpha}}}\right\rangle{\dd}x{\dd}t\\
&=\sup_{\|\gamma\|_{L^{\alpha'}\!(Q_T;E_{p/\alpha})}=1}\ \int_{Q_T}
\int_{A \in \mathbb{R}^{m \times n}} \frac{\gamma(t,x,A)(1+|A|^{\frac{p}{\alpha}})}{1+|A|^{\frac{p}{\alpha}}}\dd
\nu^{h,\epsilon}_{t,x}(A) {\dd}x{\dd}t\\
&\leq \sup_{\|\gamma\|_{L^{\alpha'}\!(Q_T;E_{p/\alpha})}=1} \int_{Q_T} \|\gamma(t,x,\cdot)\|_{E_{p/\alpha}}
|\langle \nu^{h,\epsilon}_{t,x},\tilde{a}\rangle| {\dd}x{\dd}t\\
&\leq \|\langle \nu^{h,\epsilon},\tilde{a}\rangle\|_{L^\alpha(Q_T)},
\end{align*}
which we can further bound from above by a positive constant $C$, independent of $h \in (0,h_0]$ and $\epsilon \in (0,1]$, using \eqref{e:nu-a-bound}. Hence, for each $\epsilon \in (0,1]$ there is a subsequence of the sequence of measures $\{\nu^{h,\epsilon}\}$ indexed by $h_j$ and a measure $\nu^\epsilon\in L^\alpha(Q_T;E_{p/\alpha}')$ such that for every $\gamma\in E_{p/\alpha}$ we have that
\[\langle \nu^{h_j,\epsilon},\gamma\rangle\rightharpoonup \langle \nu^\epsilon,\gamma\rangle\]
in $L^\alpha(Q_T;\R{m\times n})$ as $j\rightarrow\infty$. In particular this holds for $\gamma=a \in E_{p/\alpha}$, $1<\alpha<\frac{p}{q-1}$, that is:
\[\langle \nu^{h_j,\epsilon},a\rangle\rightharpoonup \langle \nu^\epsilon,a\rangle\]
in $L^\alpha(Q_T;\R{m\times n})$ for any such $\alpha>1$. Similar calculations hold for the family of measures $\mu^{h,\epsilon}$.

Taking the final limit in $\epsilon$ is now straightforward, and is done similarly to the above. In particular we note that we can in fact take a diagonal subsequence of $(h,\epsilon)$ and pass to the limit along this subsequence, obtaining the existence of a two measures $\mu$ and $\nu$ satisfying, for all $\psi \in W^{1,1}_0(0,T;W^{1,\alpha'}_0(\Omega;\R{m}))$ (contained in $W^{1,1}_0(0,T;W^{1,\min\{\hat{q},\alpha'\}}_0(\Omega;\R{m}))\subset L^{\min\{\hat{q},\alpha'\}}(0,T;W^{1,\min\{\hat{q},\alpha'\}}_0(\Omega;\R{m}))$):
\begin{align}\label{e:step2-end}
\int_0^T\int_\Omega \langle\mu_{t,x},\xi\rangle\cdot(\partial_t\psi(t,x)-B^{\rm T}\psi(t,x))-\langle\nu_{t,x},a\rangle : D\psi(t,x)+F(t,x)\cdot\psi(t,x){\dd}x{\dd}t=0.
\end{align}

\medskip

\noindent \textbf{{Step 3: Interpretation as a Young measure solution.}} Next, we discuss how the pair of measures $(\mu,\nu)$ satisfying \eqref{e:step2-end} can be can be interpreted as a Young measure solution. Recall that $\mu^{h,\epsilon}$ is the law of the random variable $u^h=u^{h,\epsilon}(\omega;t,x)$, and that, for each $\omega \in \mathfrak{X}$ and each $t \in [0,T]$, the function $u^{h,\epsilon}(\omega;t,\cdot) \in V^h_m \subset W^{1,\infty}_0(\Omega;\R{m})$ is a linear combination of finite element basis functions over the domain $\Omega$. Therefore,
\[x \in \Omega \mapsto \langle \mu^{h,\epsilon}_{t,x},\xi\rangle=\int_\mathfrak{X} u^{h,\epsilon}(\omega;t,x){\dd}P(\omega), \qquad t \in [0,T],\]
is weakly differentiable with respect to $x \in \Omega$, and thanks to the definition of weak derivative and Fubini's theorem,
\[D\langle \mu^{h,\epsilon}_{t,x},\xi\rangle=\int_\mathfrak{X}Du^{h,\epsilon}(\omega;t,x){\dd}P(\omega)=\langle\nu^{h,\epsilon}_{t,x},\xi\rangle,\qquad (t,x) \in (0,T] \times \Omega.\]
Using the weak-star convergence of $\mu^{h,\epsilon}$ and $\nu^{h,\epsilon}$ to, respectively, $\mu$ and $\nu$ established in Step 2, in conjunction with the regularity of $u^{h,\epsilon}$ coming from \eqref{approxreg1}--\eqref{approxreg3}, and the fact that $\mathfrak{X}$ is a space of finite $P$-measure, we see that by passing $h,\epsilon\rightarrow0_+$ along subsequences, we get the existence of a function
\[u\in L^\infty\left(\mathfrak{X};\bigtimes_{i=1}^mL^{p_i}(0,T;W_0^{1,p_i}(\Omega;\R{m}))\right)\]
such that, for $(t,x) \in (0,T] \times \Omega$,
\[\langle\mu_{t,x},\xi\rangle=\int_\mathfrak{X} u(\omega;t,x){\dd}P(\omega)\]
and
\[\langle\nu_{t,x},\xi\rangle=\int_\mathfrak{X} Du(\omega;t,x){\dd}P(\omega).\]
This comes from the fact that, for example,
\[\langle\mu^{h,\epsilon},\xi\rangle\rightharpoonup\langle\mu,\xi\rangle,\]
and we have that
\[\int_0^T\int_\Omega\int_\mathfrak{X}u^{h,\epsilon}(\omega;t,x)\cdot\phi(\omega;t,x){\dd}P(\omega){\dd}x{\dd}t
\rightarrow\int_0^T\int_\Omega\int_\mathfrak{X}u(\omega;t,x)\cdot\phi(\omega;t,x){\dd}P(\omega){\dd}x{\dd}t\]
for all $\phi\in L^1(\mathfrak{X}; L^\alpha(0,T; W_0^{1,\alpha}(\Omega;\R{m})))$. However, as $\mathfrak{X}$ has finite $P$-measure, we can choose test functions which are constant in the parameter $\omega$ to see that %
\[\int_0^T\int_\Omega\left(\int_\mathfrak{X}u^{h,\epsilon}(\omega;t,x){\dd}P(\omega)\right)\cdot\phi(t,x){\dd}x{\dd}t
\rightarrow\int_0^T\int_\Omega\left(\int_\mathfrak{X}u(\omega;t,x){\dd}P(\omega)\right)\cdot\phi(t,x){\dd}x{\dd}t\]
for all $\phi\in L^\alpha(0,T;W_0^{1,\alpha}(\Omega;\R{m}))$, which gives the desired representation for $\langle\mu_{t,x},\xi\rangle$ (the corresponding term for $\langle \nu_{t,x},\xi\rangle$ follows similarly).

Now we define the function
\[U(t,x):=\int_\mathfrak{X} u(\omega;t,x){\dd}P(\omega)=\langle\mu_{t,x},\xi\rangle.\]
The function $U$ is locally integrable over $Q_T$ and it therefore has a well-defined distributional derivative $DU$. If we can demonstrate that $DU(t,x)=\langle \nu_{t,x},\xi \rangle$ in a suitable sense, then we will have shown that the pairing $(U,\nu)$ is a Young measure solution of the system \eqref{modelprob1}--\eqref{modelconstants} under consideration. What must first be shown is that it makes sense to speak of $DU$ as a function, rather than as a distribution. To that end, note that the function $u \in L^p(0,T;W^{1,p}_0(\Omega;\R{m}))$ satisfies $u|_{\partial\Omega}=0$, and so we can extend the function $u$ (and hence the function $U$) by zero from $[0,T]\times\overline{\Omega}$ to the whole of $[0,T]\times\R{n}$ so that the extended
function (still denoted by $u$) belongs to $L^p(0,T;W^{1,p}(\R{n};\R{m}))$. We consider the difference quotient of $U$, defined by
\[D^\delta_i U=\frac{U(t,x+\delta e_i)-U(t,x)}{\delta},\]
where $e_i$ is the unit vector in the $i$-th co-ordinate direction. We then have that
\begin{align*}\|D^\delta_i U(t,\cdot)\|_{L^{p_i}(\Omega;\R{m})}&=\left\|\frac{U(t,\cdot+\delta e_i)-U(t,\cdot)}{\delta}\right\|_{L^{p_i}(\Omega;\R{m})}\\
&=\left\|\int_\mathfrak{X}\frac{u(\omega;t,\cdot+\delta e_i)-u(\omega;t,\cdot)}{\delta}{\dd}P(\omega)\right\|_{L^{p_i}(\Omega;\R{m})}\\
&\leq \int_\mathfrak{X}\left\|\frac{u(\omega;t,\cdot+\delta e_i)-u(\omega;t,\cdot)}{\delta}\right\|_{L^{p_i}(\Omega;\R{m})}{\dd}P(\omega)\\
&\leq \int_\mathfrak{X} \|D_iu(\omega;t,\cdot)\|_{L^{p_i}(\Omega;\R{m})}{\dd}P(\omega), \qquad \forall\,\delta \in (0,1), \quad  i=1,\dots,n.
\end{align*}
where in the transition to the last line we have used that, by Jensen's inequality,
\begin{align*}
&\left\|\frac{u(\omega;t,\cdot+\delta e_i)-u(\omega;t,\cdot)}{\delta}\right\|_{L^{p_i}(\Omega;\R{m})}^{p_i}
\!= \left\|\int_0^1 D_i u(\omega;t,\cdot+s\delta) \dd s\right\|^{p_i}_{L^{p_i}(\Omega;\R{m})} \!\leq \int_0^1 \left\| D_i u(\omega;t,\cdot+s\delta)\right\|^{p_i}_{L^{p_i}(\Omega;\R{m})} \dd s\\
&\qquad= \int_0^1 \int_\Omega |D_i u(\omega;t,x+s\delta)|^{p_i} \dd x \dd s \leq \int_0^1 \int_\Omega |D_i u(\omega;t,x)|^{p_i} \dd x \dd s = \|D_i u(\omega;t,\cdot)\|^{p_i}_{L^{p_i}(\Omega;\R{m})}.
\end{align*}
Therefore, by a standard characterization of Sobolev functions in terms of difference quotients, $D_iU(t,\cdot)\in L^{p_i}(\Omega;\R{m \times n})$, $i=1,\dots,n$, for almost every $t\in[0,T]$. By integrating over $t \in [0,T]$ in the above inequality it follows that $D_iU\in L^{p_i}(Q_T;\R{m \times n})$, $i=1,\dots,n$, and in particular $U \in L^p(0,T;W^{1,p}_0(\Omega;\R{m}))$, as required. By an analogous argument,
\[ U \in L^\infty(0,T;L^2(\Omega;\R{m})),\qquad \partial_t U \in L^{\hat{q}'}\!(0,T;W^{-1,\hat{q}'}\!(\Omega;\R{m})),\]
thanks to the regularity results
\[
u(\omega;\cdot,\cdot) \in L^\infty(0,T;L^2(\Omega;\R{m})),\qquad \partial_t u(\omega;\cdot,\cdot) \in L^{\hat{q}'}\!(0,T;W^{-1,\hat{q}'}\!(\Omega;\R{m})),\qquad \mbox{for a.e. $\omega \in \mathfrak{X}$,}\]
which follow from the regularity properties of $u^{h,\epsilon}$ stated in  \eqref{approxreg1}--\eqref{approxreg3} using the weak lower-semi-continuity of the norm function.

Next, we write
\[V(t,x):=\langle\nu_{t,x},\xi\rangle=\int_\mathfrak{X} Du(\omega;t,x){\dd}P(\omega),\]
and let $\varphi \in C^\infty_0(Q_T;\R{m \times n})$ be a test function. We then compute, using Fubini's Theorem, with $D = D_x$ and $\mbox{div} = \mbox{div}_x$,
\begin{align*}
\int_0^T\int_\Omega DU(t,x) : \varphi(t,x){\dd}x{\dd}t&=-\int_0^T\int_\Omega U(t,x) \cdot \mbox{div}\,\varphi(t,x){\dd}x{\dd}t\\
&=-\int_0^T\int_\Omega \int_\mathfrak{X} u(\omega;t,x) \cdot \mbox{div}\,\varphi(t,x){\dd}P(\omega){\dd}x{\dd}t\\
&=-\int_\mathfrak{X}\int_0^T\int_\Omega u(\omega;t,x) \cdot \mbox{div}\,\varphi(t,x){\dd}x{\dd}t{\dd}P(\omega)\\
&=\int_\mathfrak{X}\int_0^T\int_\Omega Du(\omega;t,x) : \varphi(t,x){\dd}x{\dd}t{\dd}P(\omega)\\
&=\int_0^T\int_\Omega\int_\mathfrak{X} Du(\omega;t,x) : \varphi(t,x){\dd}P(\omega){\dd}x{\dd}t\\
&=\int_0^T\int_\Omega V(t,x) : \varphi(t,x){\dd}x{\dd}t \qquad \forall\, \varphi \in C^\infty_0(Q_T;\R{m \times n}).
\end{align*}
Hence, $DU(t,x) = V$, and because $V:= \langle \nu, \xi \rangle$, it follows that $DU = \langle \nu, \xi \rangle$, as desired. Thus we have shown that the pair $(U,\nu)$ is a Young measure solution of the system \eqref{modelprob1}--\eqref{modelconstants} under consideration.

\medskip

\noindent \textbf{{Step 4: Attainment of the initial condition.}} As
\[\partial_tU^{h,\epsilon}\in L^{\hat{q}'}\!(0,T;W^{-1,\hat{q}'}\!(\Omega;\R{m})),\]
and
\[U^{h,\epsilon}\in \bigtimes_{i=1}^m L^{p_i}(0,T;W_0^{1,p_i}(\Omega;\R{m}))\cap L^\infty(0,T;L^2(\Omega;\R{m})),\]
it follows from Lemma \ref{ArzAscApp} that there is a subsequence $(h_j,\epsilon_j)$ of $(h,\epsilon)$ along which
\begin{align}\label{e:unicon}
\int_\Omega U^{h_j,\epsilon_j}(t,x)\cdot\varphi(x){\dd}x\rightarrow\int_\Omega U(t,x)\cdot\varphi(x){\dd}x
\end{align}
uniformly in $C([0,T])$ for all $\varphi\in L^2(\Omega;\R{m})$. Therefore to show that $U(0,\cdot) = u_0$ we let $\varphi\in L^2(\Omega;\R{m})$ and compute:
\begin{align}\begin{aligned}\label{e:1234}
\int_\Omega (U(0,x)-u_0(x))\cdot\varphi(x){\dd}x&=\int_\Omega (U(0,x)-U^{h_j,\epsilon_j}(0,x))\cdot\varphi(x){\dd}x\\
&\ \ \ \  + \int_\Omega (U^{h_j,\epsilon_j}(0,x)-U^{h_j,\epsilon_j}(t,x))\cdot\varphi(x){\dd}x\\
&\ \ \ \  + \int_\Omega \left(U^{h_j,\epsilon_j}(t,x)-\int_\mathfrak{X}u_0^{h_j,\epsilon_j}(\omega,x){\dd}P(\omega)\right)\cdot\varphi(x){\dd}x\\
&\ \ \ \  + \int_\Omega \left(\int_\mathfrak{X}u_0^{h_j,\epsilon_j}(\omega,x){\dd}P(\omega)-u_0(x)\right)\cdot\varphi(x){\dd}x\\
&=:I_j+II_j(t)+III_j(t)+IV_j.
\end{aligned}
\end{align}
The term $I_j$ converges to zero as $j\rightarrow\infty$ by the uniform convergence in $C([0,T])$ stated in \eqref{e:unicon}. The term $II_j(t)$ converges to zero as $t\rightarrow0_+$ as it holds that $u^{h_j,\epsilon_j}(\omega;t,x)\rightarrow u^{h_j,\epsilon_j}(\omega,0,x)$ by continuity of $t \in [0,T] \mapsto u^{h_j,\epsilon_j}(\omega;t,x)$, $\omega \in \mathfrak{X}$, $x \in \Omega$, and we can apply the Dominated Convergence Theorem to interchange this limit with the integral over $\mathfrak{X}$ that appears in the definition of $U^{h_j,\epsilon_j}$. The term $III_j(t)$ also converges to zero as $t\rightarrow0_+$ as we have that the numerical functions $u^{h_j,\epsilon_j}$ satisfy the initial condition, and we may once again apply the Dominated Convergence Theorem to interchange the limit with the integral over $\mathfrak{X}$. Finally for the term $IV_j$ we compute:
\begin{align*}
\left|\int_\mathfrak{X} u_0^{h_j,\epsilon_j}(\omega,x){\dd}P(\omega)-u_0(x)\right|
&=\left|\int_\mathfrak{X}u_0^{h_j}+\epsilon_j\upsilon^{h_j}(\omega;x)-u_0(x){\dd}P(\omega)\right|\\
&\leq |u_0^h(x)-u_0(x)|+\epsilon_j\int_\mathfrak{X}|\upsilon^{h_j}(\omega,x)|{\dd}P(\omega).
\end{align*}
Integrating this over $\Omega$ we get
\begin{align*}
\int_\Omega\left|\int_\mathfrak{X} u_0^{h_j,\epsilon_j}(\omega,x){\dd}P(\omega)-u_0(x)\right|{\dd}x&\leq \|u_0^{h_j}-u_0\|_{L^1(\Omega;\R{m})}+\epsilon_j \int_\mathfrak{X} \|\upsilon^{h_j}(\omega,\cdot)\|_{L^1(\Omega;\R{m})}{\dd}P(\omega)\\
&\leq \|u_0^{h_j}-u_0\|_{L^2(\Omega;\R{m})}+\epsilon_j |\Omega|^{\frac{1}{2}}\int_\mathfrak{X} \|\upsilon^{h_j}(\omega,\cdot)\|_{L^2(\Omega;\R{m})}{\dd}P(\omega)\\
&\leq \|u_0^{h_j}-u_0\|_{L^2(\Omega;\R{m})}+\epsilon_j |\Omega|^{\frac{1}{2}},
\end{align*}
which converges to zero as $j\rightarrow\infty$ by the assumed strong convergence of our discretized initial condition $u_0^h$ to $u_0$ in $L^2(\Omega;\R{m})$. In summary then,  $\lim_{j\rightarrow \infty} I_j =0$, $\lim_{t\rightarrow 0_+} II_j(t)=0$ for all $j \geq 1$,  $\lim_{t\rightarrow 0_+} III_j(t)=0$ for all $j \geq 1$, and  $\lim_{j\rightarrow \infty} IV_j =0$. Passing to the limit $t \rightarrow 0_+$ in \eqref{e:1234} with $j \geq 1$ kept fixed and then with $j \rightarrow \infty$ we deduce that
\[ \int_\Omega (U(0,x)-u_0(x))\cdot\varphi(x){\dd}x = 0 \qquad \forall\, \varphi \in L^2(\Omega;\R{m}).\]
Thus we have that $U(0,x)=u_0(x)$ for almost every $x\in\Omega$, showing that the initial condition is satisfied.

In summary, we have that that the pair $(U,\nu)$ is a Young measure solution to the model problem satisfying the initial
condition $U(0,\cdot) = u_0(\cdot)$. For clarity we state the following theorem which collects the most relevant properties
proved in Steps 1--4 above.

\begin{theorem}
Algorithms C and D converge to a Young measure solution of the problem \eqref{modelprob1}--\eqref{modelconstants} in the event that the initial Young measure is a Dirac mass, $\delta_{u_0(x)}$, concentrated on an initial datum $u_0\in L^2(\Omega;\R{m})$.
\end{theorem}
\begin{rem}
If we have a (nonatomic) measure-valued initial datum $\sigma$, Steps 1--4 above provide an outline of an algorithm for proving the existence of a Young measure solution, assuming that all of the $M$  ``perturbations" to $\sigma$ can be chosen to be bounded, independent of $M$. This requires being able to find a random variable $u_0$ whose law is given by the initial Young measure $\sigma$. The existence of such a random variable is guaranteed by Proposition \ref{randvar}.
\end{rem}
\begin{rem} In \cite{FKMT} and \cite{FMT}, in their analysis of semidiscrete schemes the authors are able to interchange the order in which the limits of the spatial discretization parameter and the parameter $\epsilon$ tending to zero and $M\rightarrow \infty$ are taken. For us here the passages to the limits in $h\rightarrow 0_+$, $\epsilon\rightarrow 0_+$ and $\Delta t \rightarrow 0_+$ are all interchangeable, but the passage in $M$ is problematic from this point of view. The reason for this is that, while the limits obtained in the proof of Theorem \ref{subsequenceconvergence} are independent of $h$, computing these limits in the first place required us to use the independence and the ``identicalness" of the distributions of the functions $u^{h,k,\epsilon}$. This required that the solution operator $\textbf{S}_t$ was measurable. However, for the model problem under consideration here, we were only able to show continuity of this operator on the functions $u^{h,k,\epsilon}$ for each fixed $h$, and were unable to transfer this property to the limiting function as we pass $h\rightarrow0_+$, meaning that within the context of the algorithms, we \emph{must} take a limit passage in $M$ before the limit passage in $h$. If instead one could show that the operator $\textbf{S}_t$ was measurable when considered as a mapping from the nondiscretized initial condition $u_0\in L^2(\Omega;\R{m})$, we would then be able to interchange the limits in any way we desired to, as this property would be sufficient for the independence and ``identicalness" of the distributions to carry over to the limiting functions as we take $h\rightarrow0_+$.
\end{rem}

\noindent\textbf{Acknowledgement.} We are grateful to Professor Mike Cullen (UK Met Office, Exeter) for stimulating discussions and for suggesting to us the problem studied in this paper. This work was supported by the UK Engineering and Physical Sciences Research Council [EP/L015811/1].

\bibliographystyle{plain}
\bibliography{paperdraft}
\end{document}